\documentclass[12pt]{amsart}
\usepackage{amssymb}
\usepackage{graphicx}
\usepackage{amsfonts}
\usepackage{amsmath}
\usepackage{amsthm}
\usepackage{fancyhdr}
\usepackage{indentfirst}
\usepackage{hyperref}
\usepackage{comment}
\usepackage{color}
\usepackage{subcaption}
\usepackage{epsfig}
\usepackage{pst-grad} 
\usepackage{pst-plot} 
\usepackage[space]{grffile} 
\usepackage{etoolbox} 
\usepackage{mathrsfs} 
\usepackage{enumitem} 
\usepackage{marginnote}

\newcommand{\al}{\alpha}

\newcommand{\R}{\mathbb{R}}
\newcommand{\N}{\mathbb{N}}

\newcommand{\Si}{\Sigma}
\newcommand{\de}{\delta}

\newcommand{\ep}{\epsilon}

\newcommand{\Ga}{\Gamma}
\newcommand{\ga}{\gamma}

\newcommand{\lap}{\triangle}

\newcommand{\rom}[1]{\expandafter\romannumeral #1}
\newcommand{\Rom}[1]{\uppercase\expandafter{\romannumeral #1}}

\newcommand{\dist}{\operatorname{dist}}
\newcommand{\Div}{\operatorname{div}}
\newcommand{\Tr}{\operatorname{Tr}}

\newcommand{\Hess}{\operatorname{Hess}}

\newcommand{\Ric}{\operatorname{Ric}}
\newcommand{\Area}{\operatorname{Area}}

\newcommand{\interior}{\operatorname{int}}
\newcommand{\ind}{\operatorname{Index}}

\newcommand{\Lip}{\operatorname{Lip}}

\usepackage{geometry}
\begin{document}

\newtheorem{theorem}{Theorem}[section]
\newtheorem{proposition}[theorem]{Proposition}
\newtheorem{corollary}[theorem]{Corollary}

\newtheorem{claim}{Claim}

\theoremstyle{remark}
\newtheorem{remark}[theorem]{Remark}

\theoremstyle{definition}
\newtheorem{definition}[theorem]{Definition}

\theoremstyle{plain}
\newtheorem{lemma}[theorem]{Lemma}
\newtheorem*{lemma1}{Lemma}

\numberwithin{equation}{section}

\title[Generic scarring for minimal hypersurfaces]{Generic scarring for minimal hypersurfaces along stable hypersurfaces}

\author[Antoine Song]{Antoine Song}
\address{Department of Mathematics, University of California, Berkeley, Berkeley, CA 94720, USA}
\email{aysong@berkeley.edu}

\author[Xin Zhou]{Xin Zhou}
\address{Department of Mathematics, Cornell University, Ithaca, NY 14853, USA, and Department of Mathematics, University of California Santa Barbara, Santa Barbara, CA 93106, USA}
\email{xinzhou@cornell.edu}


\begin{abstract}
Let $M^{n+1}$ be a closed manifold of dimension $3\leq n+1\leq 7$. We show that for a $C^\infty$-generic metric $g$ on $M$, to any connected, closed, embedded, $2$-sided, stable, minimal hypersurface $S\subset (M,g)$ corresponds a sequence of closed, embedded, minimal hypersurfaces $\{\Sigma_k\}$ scarring along $S$, in the sense that the area and Morse index of $\Sigma_k$ both diverge to infinity and, when properly renormalized, $\Sigma_k$ converges to $S$ as varifolds. We also show that scarring of immersed minimal surfaces along stable surfaces occurs in most closed Riemannian $3$-manifods. 
\end{abstract}

\maketitle

\pdfbookmark[0]{}{beg}

\setcounter{section}{-1}

\section{Introduction}
\label{S:intro}

The existence theory of closed minimal hypersurfaces has enjoyed important advancements in recent years. However there are only few general results describing the possible spatial distributions of minimal hypersurfaces in closed manifolds. The Lawson surfaces \cite{Lawson70} in the round $3$-sphere give examples of sequences of minimal surfaces that either get equidistributed, or ``scar'' in the sense that the minimal surfaces concentrate along a proper subset of the ambient manifold even though their areas diverge. In \cite{Marques-Neves-Song19}, Marques, Neves and the first-named author proved that generically in dimensions $3$ to $7$, there is a sequence of minimal hypersurfaces equidistributing on average. On the other hand, Colding-De Lellis \cite{Colding-DeLellis05} constructed examples of $3$-manifolds where a sequence of minimal surfaces accumulate around a stable minimal $2$-sphere. Wiygul \cite{Wiygul20} extended \cite{Kapouleas-Yang10} and found a sequence of minimal surfaces scarring along the Clifford torus in the round $\mathbb{S}^3$. In this paper, we show that generically in dimensions $3$ to $7$, scarring for minimal hypersurfaces happens as soon as there is a stable hypersurface. If $\Sigma$ is a closed hypersurface in $(M,g)$, we denote by $[\Sigma]$ the multiplicity one varifold associated to $\Sigma$ and by $||\Sigma||$ its area, or equivalently the mass of $[\Sigma]$. We need the varifold $\mathbf F$-distance defined in \cite[2.1(19)]{Pitts}. This distance $\mathbf{F}$ induces the usual varifold topology on the set of varifolds of mass at most $1$.

\begin{theorem} [Generic scarring]
\label{T:main}
Let $M^{n+1}$ be a closed manifold of dimension $3\leq n+1\leq 7$. For a $C^\infty$-generic metric $g$ on $M$ in the sense of Baire, the following holds. For any connected, closed, embedded, $2$-sided, minimal hypersurface $S\subset (M,g)$ which is stable, there is a sequence of closed, embedded, minimal hypersurfaces $\{\Sigma_k\}$ with the following properties:
\vspace{0.5em}
\begin{enumerate}[label=$(\roman*)$]
\addtolength{\itemsep}{0.7em}
\item 
\hfill
$\displaystyle \Sigma_k\cap S=\varnothing,$
\hfill\refstepcounter{equation}

\item 
\hfill
$\displaystyle \lim_{k\to\infty} ||\Sigma_k|| =\infty$,
\hfill\refstepcounter{equation}

\item 
\hfill
$\displaystyle \lim_{k\to \infty}\ind(\Sigma_k) ||\Sigma_k||^{-1} = ||S||^{-1}$,
\hfill\refstepcounter{equation}

\item 
\hfill
$\displaystyle \mathbf{F}\big(\frac{[S]}{||S||}, \frac{[\Sigma_k]}{||\Sigma_k||}\big) \leq \frac{1}{\log(||\Sigma_k||)}$. 
\hfill\refstepcounter{equation}
\end{enumerate}
\end{theorem}

\begin{remark}
This theorem confirms a conjecture by the first-named author; see \cite[Conjecture 4]{Song19}. From the proof of Theorem \ref{T:main}, it will be clear that it can be generalized to $1$-sided minimal hypersurfaces $S$ with stable $2$-sided double cover. A particular feature of Theorem \ref{T:main} is that we obtain an explicit quantitative estimate on how close the minimal hypersurfaces are from $S$. The logarithm is not sharp and can at least be replaced by any function $f(x)$ growing slower than positive powers of $x$, i.e. $\lim_{x\to \infty} f(x) x^{-\alpha}=0$ for all $\alpha>0$. 
\end{remark}

\subsection*{Motivations and historical backgrounds} 

Minimal hypersurfaces, as critical points of the area functional, can be considered as non-linear geometric analogues of the Laplacian eigenfunctions. Equidistribution and scarring have been widely studied in spectral theory and ergodic theory; we recall here some results that partly motivate the study of distributions of minimal hypersurfaces. First, pointwise Weyl laws imply that the $L^2$-densities of eigenfunctions equidistributes on average; see \cite{Avakumovic56}. In negative curvature, the quantum ergodicity theorem asserts that for a density one subsequence of eigenfunctions, the normalized densities individually equidistribute, a fact related to the ergodicity of the geodesic flow in negative curvature; see \cite{Snirelman74, Zelditch87,CdV85}. The Quantum Unique Ergodicity conjecture \cite{Rudnick-Sarnak94} predicts that the full sequence should equidistribute; see \cite{Zelditch17, Sarnak95, Hassell11} for surveys. In contrast, in certain generic situations some subsequences do not equidistribute; instead the $L^2$-densities ``scar'' along subregions of the ambient manifolds; we mention \cite{Hassell10} which treats certain ergodic billiards, and \cite{Babich-Lazutkin67,Ralston77,CdV09} which imply scarring of eigenfunctions along certain elliptic geodesics in generic closed surfaces. Analogous behaviors for eigenfunctions of certain quantized toral automorphisms were discovered in \cite{Faure-Nonnenmarcher-DeBievre03}.

As mentioned earlier, the existence of minimal hypersurfaces in closed manifolds is rather well understood by now. The Almgren-Pitts min-max theory \cite{Almgren62, Almgren65, Pitts, Schoen-Simon81} was much improved starting from the work of Marques-Neves \cite{Marques-Neves14, Marques-Neves17}. 
Yau's conjecture on the existence of infinitely many closed minimal surfaces was proved in the general case by the first-named author \cite{Song18}, building on Marques-Neves' approach \cite{Marques-Neves17}. The generic equidistribution theorem in \cite{Marques-Neves-Song19} quantifies the generic density result of Irie-Marques-Neves \cite{Irie-Marques-Neves18}. Both proofs uses the Weyl Law for the area functional by Liokumovich-Marques-Neves \cite{Liokumovich-Marques-Neves16}. Around the same time, a Morse theory was established for the area functional: the second-named author \cite{Zhou19} proved the Multiplicity One Conjecture raised by Marques-Neves \cite{Marques-Neves16, Marques-Neves18}, using theory developed in \cite{Zhou-Zhu18}; (see also Chodosh-Mantoulidis \cite{Chodosh-Mantoulidis20}). When combined with \cite{Marques-Neves18}, this implies that, for bumpy metrics, there is a closed embedded minimal hypersurface of Morse index $p$ for every $p\in\N$. Recently, the Morse inequalities for the area functional were established for bumpy metrics by Montezuma-Marques-Neves in \cite{Marques-Montezuma-Neves20}.

\subsection*{More on dimension 3}

In dimension $n+1=3$, more can be said about the scarring minimal surfaces $\Sigma_k$ produced by Theorem \ref{T:main}. They satisfy for instance a local sheet accumulation property, as we explain now. Given a positive integer $N_0$, the first-named author introduced in \cite{Song19} a two-piece decomposition for any minimal surface $\Sigma$ into a non-sheeted part $\Sigma_{\leq N_0}$ and sheeted part $\Sigma_{>N_0}$:
$$\Sigma=\Sigma_{\leq N_0} \sqcup \Sigma_{>N_0},$$
in analogy with the usual thick-thin decomposition for manifolds with bounded sectional curvature. Roughly speaking, at each point of $\Sigma_{\leq N_0}$ (resp. $\Sigma_{> N_0}$), the number of sheets at the scale of stability, or equivalently scale of curvature, is at most $N_0$ (resp. larger than $N_0$). It was proved that the genus and area of $\Sigma_{\leq N_0} $ are always controlled by the index of $\Sigma$, independently of the area of $\Sigma$, see \cite[Section 4]{Song19}. For the minimal surfaces $\Sigma_k$ in Theorem \ref{T:main}, it follows from \cite[Theorem 17]{Song19} that they are mostly locally sheeted in the sense that for any integer $N_0$,
$$\lim_{k\to\infty} \frac{\Area(\Sigma_{k,\leq N_0})}{\Area(\Sigma_k)}=0.$$

Moreover, in dimension $n+1=3$, topology often helps finding stable minimal surfaces. Combining this observation with the proof of Theorem \ref{T:main}, we show the following. 
\begin{theorem} [Generic scarring in dimension 3]
\label{T:3dimension}
Let $M^3$ be a closed $3$-manifold which is not diffeomorphic to a spherical quotient $\mathbb{S}^3/\Gamma$. Then for a $C^\infty$-generic metric $g$ on $M$, there is a sequence of connected, closed, immersed, minimal surfaces $\{\Sigma_k\}$ such that their Morse index and the area of their images both diverge to infinity, and 
$$ \lim_{k\to 0}\mathbf{F}\big(\frac{[S]}{||S||},\frac{[\Sigma_k]}{||\Sigma_k||}\big)=0$$
for some connected, closed, immersed, stable minimal surface $S\subset (M, g)$.
\end{theorem}

In the case $M$ carries a hyperbolic metric, one can say more: given any $\pi_1$-injective surface $F$, there is a sequence of minimal surfaces scarring along an area minimizing surface homotopic to $F$; see Corollary \ref{C:hyperbolic}. It is known that $\pi_1$-injective surfaces exist in abundance by Kahn-Markovi\'{c} \cite{Kahn-Markovic12a, Kahn-Markovic12b}. We point out that the asymptotic growth of the size of certain families of area minimizing surfaces in negatively curved closed $3$-manifolds was recently studied by Calegari-Marques-Neves \cite{Calegari-Marques-Neves20}.

\begin{remark}
Compared to Theorem \ref{T:main}, we can ensure in Theorem \ref{T:3dimension} that $\Sigma_k$ is connected, on the other we cannot give quantitative estimates relating the index and area of $\Sigma_k$. Each surface  $\Sigma_k$ may be a multiple cover of its image, but the degree of the covering is uniformly bounded as $k\to \infty$.
\end{remark}

\subsection*{Overview of Proofs}

The proof of Theorem \ref{T:main} is based on a perturbation argument applied to an asymptotic formula for certain min-max widths. In \cite{Irie-Marques-Neves18}, such a strategy was applied to the Weyl law for the volume spectrum \cite{Liokumovich-Marques-Neves16} in order to get generic density of minimal hypersurfaces; (see \cite{Song19a} for an alternative argument bypassing the Weyl law). This result was quantified in \cite{Marques-Neves-Song19}. It is such a quantified argument that we wish to use here, but with another Weyl law type formula introduced by the first-named author in the proof of Yau's conjecture \cite{Song18}. This formula, that we will call cylindrical Weyl law (see Section \ref{S:core manifold}), describes the growth of a sequence of min-max numbers $\{\tilde{\omega}_p\}$ canonically associated with a given strictly stable minimal hypersurface $S$:
\[ \lim_{p\to \infty}\frac{\tilde{\omega}_p}{p} = \Area(S). \]
These min-max numbers $\tilde{\omega}_p$ are realized as the areas of some minimal hypersurfaces. When we perturb the metric, these areas are asymptotically only sensitive to the area of $S$ and not to the changes far from $S$, by the cylindrical Weyl law. So intuitively it should imply that ``on average'' in a neighborhood of the given metric, those minimal hypersurfaces are accumulating around $S$. Similarly to \cite{Marques-Neves-Song19}, the proof relies on a derivative estimate for the min-max numbers $\tilde{\omega}_p$.

An important difference with \cite{Marques-Neves-Song19} is that we want to obtain minimal hypersurfaces with both area and index diverging to infinity. To achieve this, it is not possible to directly apply the Multiplicity One Conjecture proved by the second-named author \cite{Zhou19} and the Morse Index Conjecture proved in \cite{Marques-Neves18}. This issue will require to study a certain codimension $1$ Banach submanifold in the space of minimal embeddings defined in \cite{White91}.

An additional feature of this cylindrical Weyl law compared to the usual Weyl law for the volume spectrum is the possibility to estimate the remainder term. This enables us, after proving a quantitative Constancy Theorem, and by carefully keeping track of all the perturbations, to prove that the varifold distances between $S$ and certain normalized minimal hypersurfaces have a quantified estimate in terms of the areas. The final scarring result follows from an induction argument on all strictly stable minimal hypersurfaces.

\vspace{1em}
The proof of Theorem \ref{T:3dimension} starts with the following observation: for most closed $3$-manifolds $M$, after taking a finite cover $M_{cover}$, one can find a non-trivial homotopy class of embedded surfaces. By a standard minimization procedure, one gets an embedded stable minimal surface in $M_{cover}$. Then one would like to apply Theorem \ref{T:main} to $M_{cover}$ and project back to $M$. The problem is that generic metrics on $M_{cover}$  are in general not lifts of metrics on $M$, so in the end we prove a scarring result without quantitative estimates on the Morse index and area.

\subsection*{Organization}

The paper is organized as follows. In Section \ref{S:Preliminaries}, we collect preliminary results on min-max theory for minimal hypersurfaces in compact manifolds with boundary, including some properties of the cylindrical volume spectrum. In Section \ref{S:constancy}, we prove a quantitative version of the Constancy Theorem. In Section \ref{S:metric deformation}, we use the perturbation argument to prove a key intermediate deformation result. Finally we finish the proof of our generic scarring theorems in Section \ref{S:scarring}, and apply those arguments to $3$-manifolds which are not spherical quotients. In Appendix \ref{A:Differentiability of the first eigenvalue}, we recall a well-known fact on the differentiability of the first eigenvalue of self-adjoint second order elliptic operators.

\vspace{1em}
{\bf Acknowledgement:} This research was partially conducted during the period A. S. served as a Clay Research Fellow. X. Z. is partially supported by NSF grant DMS-1811293, DMS-1945178, 
and an Alfred P. Sloan Research Fellowship. We would like to thank Peter Sarnak for discussions and for pointing out \cite{Babich-Lazutkin67,Ralston77}.

\section{Preliminaries}
\label{S:Preliminaries}

\subsection{Min-max theory in compact manifolds with boundary}
\label{S:core manifold}

In this part, we describe results revolving around a procedure to construct closed minimal hypersurfaces in a compact manifold $\hat{M}$ with strictly stable minimal boundary. This procedure was introduced by the first-named author in \cite{Song18} in relation to Yau's conjecture. The idea is to attach a cylindrical end to each connected boundary component of $\hat{M}$ to obtain a non-compact manifold with a Lipschitz metric, and then study the sequence of its min-max widths (or volume spectrum) \cite{Marques-Neves17}. It was proved that these widths grows linearly, with leading coefficient the largest area among areas of the boundary connected components. Moreover these widths are realized by the areas (with multiplicities) of closed minimal hypersurfaces embedded in the interior of $\hat{M}$, and it was shown by observing that the free boundary min-max theory of M. M.-C. Li and the second-named author \cite{Li-Zhou16} applied to compact approximations of the non-compact manifold produces compact embedded minimal hypersurfaces that eventually pull back to the interior of the original manifold $\hat{M}$. These results were recently adapted by Y. Li \cite{Y.Li20} who showed, based on the resolution of the Multiplicity One Conjecture by the second-named author \cite{Zhou19} and the Morse Index Conjecture by Marques-Neves \cite{Marques-Neves18}, that in a bumpy metric, the constructed minimal hypersurfaces have multiplicity one and linear growing Morse index.

Let $(\hat{M}^{n+1}, \partial\hat{M}, g)$ denote a compact connected $(n+1)$-dimensional smooth manifold with boundary endowed with a $C^q
$ metric with $q\geq 3$. Assume that $\partial \hat M =\cup_{i=1}^l S_i$ is a disjoint union of connected, strictly stable, minimal hypersurfaces $\{S_i\}_{i=1}^l$.  The metric $g$ is said to be {\em embedded bumpy} (resp. {\em immersed bumpy}) if every closed embedded (resp. immersed) minimal hypersurface inside $\hat M$ is non-degenerate, i.e. has no non-zero Jacobi fields.  It is useful to introduce the following manifold with cylindrical ends \cite{Song18}:
\begin{equation}
\mathcal C(\hat M) = \hat M \cup_{id} (\partial\hat M \times [0, \infty)),
\end{equation}
where $id: \partial \hat M \times \{0\}\to \partial \hat M$ is the canonical identity map. Note that $\mathcal C(\hat M)$ is a non-compact smooth manifold. $\mathcal C(\hat M)$ is endowed with a natural Lipschitz metric $h$ by simply putting together $g$ with the product metric $g|_{\partial\hat M}+(dt)^2$.

Consider an exhaustion of $\mathcal C(\hat M)$ by compact subsets $K_1\subset K_2\subset\cdots \subset \mathcal C(\hat M)$. The following definition does not depend on the particular choice of exhaustion $K_1\subset K_2\subset\cdots\subset \mathcal C(\hat M)$.
\begin{definition}[\cite{Song18}]
\label{D:cylindrical width}
For each positive integer $p$, the {\em cylindrical $p$-width} of $(\hat M, g)$ is defined as
\begin{equation}
\tilde\omega_p(\hat M, g) :=  \omega_p(\mathcal C(\hat M), h) := \lim_{i\to\infty} \omega_p(K_i, h).
\end{equation}
Here $\omega_p(K_i, h)$ is the $p$-width of $(K_i, h)$ defined in \cite[Definition 4.3]{Marques-Neves17}.
The sequence $\{\tilde\omega_p(\hat M, g)\}_{p}$ will be called {\em cylindrical volume spectrum}.
\end{definition}

These widths satisfy a ``cylindrical'' Weyl Law; see \cite[Theorem 8]{Song18}; (compare with the classical Weyl Law \cite{Liokumovich-Marques-Neves16}).
\begin{theorem}
\label{T:cylindrical Weyl Law}

Suppose that $S_1$ has maximal area among the boundary components $\{S_i\}_{i=1}^l$. Then there is a constant $C=C(g)$, such that 
\begin{equation}
p\cdot \Area(S_1) \leq \tilde\omega_p(\hat M, g) \leq p\cdot \Area(S_1) + C(g) p^{\frac{1}{n+1}}.
\end{equation}
Moreover, the constant $C(g)$ can be chosen to be locally bounded in the set of $C^q$ metrics endowed with the $C^1$ topology. 
\end{theorem}

The fact that $C(g)$ can be chosen to be bounded in a $C^1$-neighborhood of $g$ follows from the proof of \cite[Theorem 8]{Song18}, where it is apparent that this $C(g)$ is (up to a universal factor) given by the constant in \cite[Theorem 5.1]{Marques-Neves17} applied to the compact manifold $(\hat{M}, g)$; (see \cite{Gromov88,Guth09} for the original versions of this result). That constant in turn comes from bounds on a cubical complex $K$ and a Lipschitz homeomorphism $G:K\to \hat{M}$ and so depends only on a $C^1$-neighborhood of $g$.

The following min-max theorem is an extension of \cite[Theorem 9]{Song18} by Y. Li \cite[Theorem 5]{Y.Li20} using the Multiplicity One Conjecture \cite{Zhou19} and the Morse Index Conjecture \cite{Marques-Neves16, Marques-Neves18}. 

\begin{theorem}
\label{T:min-max theory for cylindrical width}
Assume that $3\leq n+1\leq 7$. Then for each $p\in\N$, there exists a $C^q$ closed, embedded, minimal hypersurface $\Gamma_p$ contained in $\interior(\hat M)$, whose connected components are called $\Gamma_p^{(1)},...,\Gamma_p^{(k_p)}$, and associated positive integer multiplicities $m_1,...,m_{k_p}$ such that
\begin{equation}
\begin{aligned}
&\tilde\omega_p(\hat M, g) = \sum_{i=1}^{k_p} m_i \Area(\Gamma^{(i)}_p), \quad \text{ and }\\
&\ind(\Gamma_p) \leq p.
\end{aligned}
\end{equation}
Moreover if $g$ is immersed bumpy then $\Gamma_p$ can be chosen so that each $\Gamma_p^{(i)}$ is $2$-sided and $m_i = 1$ for all $i\in \{1,...,k_p\}$, and $\ind(\Gamma_p) = p$. 
\end{theorem}

The results cited above were originally stated for smooth metrics, but the proofs directly apply to $C^q$ metrics, for $q\geq 3$.

\subsection{Cylindrical widths are locally Lipschitz functions of the metric}

Let $(M^{n+1}, g)$ be a closed connected $(n+1)$-dimensional Riemannian manifold with $3\leq (n+1)\leq 7$. For an integer $q \geq 3$ or $q=\infty$, we denote by $\mathbf{\Gamma}^{(q)}$ the set of $C^q$ metrics on $M$, endowed with the $C^q$ topology. In this paper the specific choice of $q$ will not matter, as long as it is chosen large enough. As in the last section, the metric $g$ is said to be {\em embedded bumpy} (resp. {\em immersed bumpy}) if every closed embedded (resp. immersed) minimal hypersurface in $M$ is non-degenerate. By \cite{White91,White17}, the set of embedded or immersed bumpy metrics is a $C^q$-generic subset in $\mathbf{\Gamma}^{(q)}$ for any $q \geq 3$ or $q=\infty$, in the sense of Baire category.

Fix $q \geq 3$ in this Subsection, and let $g$ be a $C^q$ metric. Let $S\subset (M, g)$ be a {connected,} closed, embedded, minimal hypersurface, which is 2-sided strictly stable. By the Implicit Function Theorem (see also \cite[Theorem 2.1]{White91}), 
given any other metric $\hat{g}$ in a small $C^q$-neighborhood of $g$, there is a unique minimal hypersurface $S_{\hat{g}}$ that can be written as a section of the normal bundle of $S$, with small $C^{j, \al}$-norm ($j\leq q-1$). In particular, $S_{\hat{g}}$ is also 2-sided and strictly stable. 

For $\hat{g}$ close to $g$ in the $C^q$ topology, we can define $\hat{M}_{\hat{g}}$ as follows. We pick a connected component of $M\backslash S_{\hat{g}}$, take the metric completion and obtain a compact Riemannian manifold with boundary $(\hat{M}_{\hat{g}}, \hat{g})$. Note that if $S_{\hat{g}}$ separates $M$, we can choose the component $\hat{M}_{\hat{g}}$ so that it depends continuously on $\hat{g}$, and in this case $\partial\hat M_{\hat{g}}=S_{\hat{g}}$. Otherwise, $S_{\hat{g}}$ does not separate $M$ and $\partial\hat{M}_{\hat{g}}$ consists of two isometric copies of $S_{\hat{g}}$. The manifold $(\hat{M}_{{g}}, {g})$ obtained for $g$ will be denoted by $\hat{M}_0$.

The following lemma says that the normalized cylindrical widths $\frac{1}{p}\tilde{\omega}_p(\hat M_{\hat{g}}, \hat g)$ are Lipschitz functions of the  metric $\hat g$ within a small $C^q$-neighborhood of $g$. Note that a similar property for the normalized classical widths was proved in \cite[Lemma 1]{Marques-Neves-Song19}; however, in our current situation,  we have to take care of the fact that the underlying manifold is changing when the metric $\hat g$ changes. The norms $||.||_{C^k}$ are measured with respect to the metric $g$.
\begin{lemma}
\label{L:cylindrical withs are Lipschitz}
Fix a positive integer $ q \geq 3$ and let $g$ be a $C^q$ metric on $M$. Assume that $S$ is a {connected}, closed, embedded, minimal hypersurface which is 2-sided and strictly stable. Then there exist $\ep>0$ and $C>0$ (depending only on $g$, $S$, and $q$), such that for any {$C^q$} metrics $\hat{g}_1$ and $\hat{g}_2$ with $\|\hat{g}_i-g\|_{C^q}<\ep$ for $i=1, 2$, we have
\[ \Big| \frac{1}{p}\tilde{\omega}_p(\hat M_{\hat{g}_1}, \hat g_1) - \frac{1}{p}\tilde{\omega}_p(\hat M_{\hat{g}_2}, \hat g_2) \Big|\leq C \|\hat g_1-\hat g_2\|_{C^q}. \]
\end{lemma}
\begin{proof}
For any {$C^q$} metric $\hat g$ with $\|\hat g- g\|_{C^q}\ll 1$, by the Implicit Function Theorem, $S_{\hat g}$ is a normal graph over $S$. Choose a unit normal $\nu$ of $S$, and denote the graphical function by $\hat u$, then $S_{\hat g} = \{\exp_p\big(\hat u(p)\nu(p)\big): p\in S\}$. Moreover, we have the following fact, proved in \cite{White91}.

\vspace{1em}
\textbf{Fact:} \textit{given $j\leq q-1$, $j\in\N$, there exist $\ep>0$ and $C>0$, such that for any two smooth metrics $\hat{g}_i$, $i=1, 2$, if $\|\hat{g}_i-g\|_{C^q}<\ep$, then we have}
\[ \|\hat u_1- \hat u_2 \|_{C^{j, \al}(S)} \leq C \|\hat g_1- \hat g_2\|_{C^q}. \]

We will construct diffeomorphisms $\Phi_i: M\to M$ which are identity away from a tubular neighborhood of $S$ and $\Phi_i$ maps $S$ to $S_{\hat{g}_i}$ respectively. Using such diffeomorphisms, we can identify $(\hat M_{\hat{g}_i}, \hat g_i)$ with $(\hat M_0, \Phi_i^* \hat g_i)$ so that the cylindrical widths are defined on a fixed underlying manifold $\hat M_0$. 

We will omit the sub-index when describing the construction of $\Phi$. Letting $t$ be the signed distance function over $S$ (see Section \ref{S:constancy} for more details), we have a Fermi coordinate system in a tubular neighborhood $N_{2\de}(S)$ of $S$ (for some $\delta>0$), such that
\[ N_{2\de}(S) = \{(x, t): x\in S, -2\de<t<2\de\}. \]
Choose a cutoff function $\phi: [-2\de, 2\de]\to [0, 1]$ such that $\phi\equiv 1$ on $[-\de, \de]$, $\phi\equiv 0$ on $[-2\de, 2\de]\setminus[-\frac{3}{2}\de, \frac{3}{2}\de]$, and $|\phi'|\leq \frac{3}{\de}$. Then $\|\phi\|_{C^{j, \al}}$ is uniformly bounded (depending only on $g$, $S$, $j$ and $\de$). Let $\hat g$ be a $C^q$ metric with $\|\hat{g}-g\|_{C^q}<\ep\ll 1$, $S_{\hat g}$ the associated minimal hypersurface arising from $S$, and $\hat u$ its graphical function over $S$. We define $\Phi=\Phi_{\hat g}: M\to M$ by
\[ 
\Phi (p) =
\left\{
\begin{aligned}
\big(x, t+\phi(t)\hat{u}(x)\big) & \quad \text{ if $p=(x, t)\in N_{2\de}(S)$ }\\
p & \quad \text{ if $p\notin N_{2\de}(S)$}
\end{aligned}
\right.. 
\]
It is easy to see (by the bound $|\phi'|\leq \frac{3}{\de}$) that $\Phi$ is a diffeomorphism when $\|\hat u\|_{L^\infty}$ is small enough, which can be satisfied by choosing $\ep$ small enough. Also $\|\Phi\|_{C^{j, \al}}$ is uniformly bounded.

We get the following estimate for the $C^0$ distance between $\Phi_1^*\hat g_1$ and $ \Phi_2^*\hat g_2$ computed with $g$:
\[
\begin{aligned}
\|\Phi_1^*\hat g_1- \Phi_2^*\hat g_2\|_{C^0}
&\leq 
\| \Phi_1^*\hat{g}_1 - \Phi_2^*\hat{g}_2 \|_{C^{j-1}} \\
& \leq  \| \Phi_1^*\hat{g}_1 - \Phi_1^*\hat{g}_2 \|_{C^{j-1}} + \| \Phi_1^*\hat{g}_2 - \Phi_2^*\hat{g}_2 \|_{C^{j-1}} \\
& \leq C \big( \| \hat{g}_1 - \hat{g}_2 \|_{C^{j-1}} + \|\Phi_1 - \Phi_2\|_{C^j} \big) \\
& \leq C \big( \| \hat{g}_1 - \hat{g}_2 \|_{C^j} + \|\hat u_1 - \hat u_2\|_{C^j} \big) \\
& \leq C \| \hat{g}_1 - \hat{g}_2 \|_{C^q}. 
\end{aligned}
\]

We are now ready to compare the cylindrical widths. Using Definition \ref{D:cylindrical width}, Theorem \ref{T:cylindrical Weyl Law}, and the same calculation as in \cite[Lemma 1]{Marques-Neves-Song19}, we obtain for any compact exhaustion $K_1\subset K_2\subset\dots\subset \mathcal{C}(\hat{M}_0)$:
\[  
\begin{aligned}
&\quad\,\, \tilde\omega_p(\hat M_{\hat{g}_1}, \hat g_1) - \tilde\omega_p(\hat M_{\hat{g}_2}, \hat g_2) \\
&  = \tilde\omega_p(\hat M_0, \Phi_1^*\hat g_1) - \tilde\omega_p(\hat M_0, \Phi_2^*\hat g_2) \\
& = \lim_{i\to \infty} \omega_p(K_i, h(\Phi_1^*\hat g_1)) - \lim_{i\to \infty} \omega_p(K_i, h(\Phi_2^*\hat g_2)) \\
& \leq \sup_i\Bigg(\Big(1+\sup_{v\neq 0, v\in TK_i}\frac{|h(\Phi_1^*\hat g_1)(v, v)- h(\Phi_2^*\hat g_2)(v, v)|}{h(\Phi_2^*\hat g_2)(v, v)} \Big)^{\frac{n}{2}} -1 \Bigg)\\
&\quad \cdot \lim_{i\to \infty}\omega_p(K_i, h(\Phi_2^*\hat g_2)) \\
& \leq \big( (1+C\|\Phi_1^*\hat g_1- \Phi_2^*\hat g_2\|_{C^0})^{\frac{n}{2}} -1\big) \tilde \omega_p(\hat M_{\hat{g}_2}, \hat g_2)\\
& \leq \big( (1+C\|\hat g_1 - \hat g_2\|_{C^q})^{\frac{n}{2}} -1\big) C \cdot  p \cdot \Area(\hat S_{\hat g_2}, \hat g_2)\\
& \leq \big( (1+C\|\hat g_1 - \hat g_2\|_{C^q})^{\frac{n}{2}} -1\big) C \cdot  p \cdot \Area(S, g)\\
& \leq C  \cdot p  \cdot\|\hat g_1-\hat g_2\|_{C^q},
\end{aligned}
\]
where $C$ depends only on $(M, g)$, $q$ and $S$. This finishes the proof.
\end{proof}

We end this section with a formula for the derivatives of the cylindrical widths whose proof is similar to that of Lemma 2 in \cite{Marques-Neves-Song19}. This formula holds at metrics satisfying a property slightly stronger than embedded bumpiness and involves multiplicity one min-max hypersurfaces, which is key to obtain the lower index bounds in Theorem \ref{T:main}.

\begin{lemma}
\label{L:derivative of cylindrical widths}
Fix a positive integer $ q \geq 3$ and let $g$ be a $C^q$ metric on $M$. Assume that $S$ is a {connected,} closed, embedded, minimal hypersurface which is 2-sided and strictly stable.  Let $\{\hat{g}_t\}_{t\in[0, 1]}$ be a smooth family of $C^q$ metrics, with $\hat{g}_0=g$. 
Assume that $g$ is an embedded bumpy metric and that for any $1$-sided, connected,
closed, embedded, minimal hypersurface in $(M,g)$, its $2$-sided double cover has no positive Jacobi field. Suppose also that, for some integer $p$, the cylindrical $p$-width function
$$t\to \tilde\omega_p(\hat M_{\hat{g}_t}, \hat{g}_t)$$ 
is differentiable at time $t=0$.

Then there exists a $C^q$ closed, embedded, minimal hypersurface $\Gamma_p$ contained in $\interior(\hat{M}_{g})$,  such that  
\[
\begin{split}
& \displaystyle \tilde{\omega}_{p}(\hat{M}_{g},g) = \Area_{g}(\Gamma_p), \quad \ind(\Gamma_p) = p,\\
\text {and}\quad &  \displaystyle \frac{d}{d t}\Big|_{t=0}\tilde{\omega}_{p}(\hat{M}_{\hat{g}_t},\hat{g}_t) = \int_{\Gamma_p}\frac{1}{2} \Tr_{\Gamma_p, g}  \big(\frac{\partial \hat{g}_t}{\partial t}\Big|_{t=0} \big)d\Gamma_p.
\end{split}
\]
\end{lemma}
\begin{proof}
Take a sequence $\{t_m\}_{m\in\N}\subset [0, 1]$ with $t_m\to 0$. 
For each $m$, we find a smooth metric $h_m$ which is immersed bumpy and is arbitrarily close to $\hat{g}_{t_m}$ in the $C^q$ topology, so that 
$$\frac{d}{d t}\Big|_{t=0}\tilde{\omega}_{p}(\hat{M}_{\hat{g}_t},\hat{g}_t) =\lim_{m\to \infty}\frac{\tilde{\omega}_{p}(\hat{M}_{h_m},h_m)-\tilde{\omega}_{p}(\hat{M}_{g},g)}{t_m}$$
$$\text{ and } \frac{\partial \hat{g}_t}{\partial t}\Big|_{t=0} = \lim_{m\to \infty}\frac{h_m-g}{t_m}.$$
By Theorem \ref{T:min-max theory for cylindrical width}, for each $m$ there is a $2$-sided, closed, embedded, minimal hypersurface $\Gamma'(m)\subset (\interior(\hat{M}_{h_m}),h_m)$ such that
$$ \displaystyle \tilde{\omega}_{p}(\hat{M}_{h_m},h_m) = \Area_{h_m}(\Gamma'(m)), \quad \ind(\Gamma'(m)) = p.$$
By compactness \cite{Sharp17}, these bounds imply that after taking a subsequence, as $m\to \infty$, $\Gamma'(m)$ converges in the varifold sense to a closed embedded minimal hypersurface  $\Gamma_p$ in $(\hat{M}_g,g)$. By standard Jacobi field arguments, if the convergence was not smooth with multiplicity one, then either there is a $2$-sided component of $\Gamma_p$ with a nontrivial Jacobi field or a $1$-sided component of $\Gamma_p$ whose $2$-sided double cover has a positive Jacobi field. In both case, there is a contradiction with the assumptions on the metric $g$. Thus the convergence to $\Gamma_p$ is smooth with multiplicity one, and hence convergence also holds in the Hausdorff topology, in particular $\Gamma_p\subset \interior(\hat{M}_g)$ (as otherwise by the monotonicity formula a connected component of $\Gamma'(m)$ would lie in a tubular neighborhood of $S_{h_m}$, contradicting strict stability of $S_{h_m}$). Clearly the index cannot drop by embedded bumpiness of $g$, so 
$$\displaystyle \tilde{\omega}_{p}(\hat{M}_{g},g) = \Area_{g}(\Gamma_p), \quad \ind(\Gamma_p) = p.$$

Finally, since $g$ is embedded bumpy, it is a regular value of the projection $\Pi$ from the Banach manifold of $C^{2,\alpha}$ minimal embeddings to the space of $C^q$ metrics, defined in \cite{White91}; (see Section \ref{S:metric deformation} for more details). 
The formula for the derivative then follows readily, as in \cite[Lemma 2]{Marques-Neves-Song19}.
\end{proof}

\section{A quantitative constancy theorem}
\label{S:constancy}

The Constancy Theorem in geometric measure theory \cite{Allard72}\cite[\S 41]{Simon83} says that if a stationary $k$-varifold is supported in a smooth embedded $k$-dimensional submanifold, then the varifold is represented by a constant multiple of the smooth submanifold. In this section, we prove a quantified version for varifolds of codimension one. In particular, we will prove (Theorem \ref{T:constancy}) that if the total mass of a stationary $n$-varifold in an $(n+1)$-dimensional manifold is mostly concentrated in a tubular neighborhood of a closed embedded hypersurface, then the varifold distance between the normalized varifold and the hypersurface has a quantitative estimate.

Let $(M^{n+1},g)$ be a closed Riemannian manifold, $g$ a $C^q$ metric ($q \geq 3$), and $S\subset M$ a  2-sided,  closed, embedded hypersurface. Choose  a unit normal vector field $\nu$ along $S$. Let $N_\ep(S)$ be an $\ep$-tubular neighborhood of $S$ for some small enough $\ep>0$, so that the normal exponential map of $S$ has no focal point in $N_\ep(S)$. Consider a foliation $\{S_t\}_{-\ep< t < \ep}$ of $N_\ep(S)$ by a family of equidistant hypersurfaces, that is:
\[ S_t := \{ \exp_p(t\nu(p)): p\in S \}. \]
Here $t: N_\ep(S)\to (-\ep, \ep)$ is the signed distance to $S$. 
Denote the nearest point projection map from $N_\ep(S)$ to $S$ by
\[ \pi: N_\ep(S) \to S. \]
The vector field $\nu:=\nabla t$ is a parallel vector field, i.e. $\nabla_\nu \nu=0$, and $\nu|_{S_t}$ is the unit normals of $S_t$ for all $t\in (-\ep, \ep)$. Note that if $\ga: (-\ep, \ep)\to N_\ep(p)$ is an integral curve of $\nu$ with $\ga(0)\in S$, then $\pi(\ga(t))=\ga(0)$ for all $t\in (-\ep, \ep)$.

Given a varifold $V$ and a domain $U$, $||V||$ is its mass and $||V||(U)$ the mass of the restriction of $V$ to $U$. {Before stating the theorem, let us define the $\mathbf{F}$-distance for varifolds. The original definition in \cite[2.1(19)]{Pitts} goes as follows: we first embed $(M,g)$ isometrically in a Euclidean space $\R^P$ endowed with the standard inner product, then the Grassmannian $G_{\R^P}(P, n)$ is equipped with a natural distance function, and the set of continuous functions with Lipschitz constant $1$ is well defined. Then if $V,W$ are two $n$-varifolds in $M$, 
$$\mathbf{F}(V,W) := \sup\{|V(f) - W(f)| ; \text{ $f:G_{\R^P}(P, n) \to \R$, $|f|\leq 1$, $\Lip(f)\leq 1$} \}.$$
We will use an essentially equivalent but more intrinsic definition, which makes it easier to see how $\mathbf{F}$ changes as the metric $g$ varies. Given a metric $g$, the Grassmannian bundle $G(n+1, n)$ of $n$-planes in $M$ is identified with a $\mathbb{Z}_2$ quotient of the unit tangent bundle $UTM$ of $M$, which has a natural metric and thus a natural distance function $\dist_g$. This is the distance function on $G(n+1, n)$ that we will use to define the Lipschitz constant of a function $f:G(n+1,n)\to \R$, and it is clearly equivalent to that defined by Pitts up to a constant depending on the isometric embedding. Note that the $\mathbf{F}$-distance induces the usual topology on the set of varifolds with mass at most $1$.}

\begin{theorem}
\label{T:constancy}
Let $(M, g)$, $S$, $\{S_t\}$, $N_\ep(S)$ be as above.
Then there exists a constant $C(g,S)$ depending only on $g$ and $S$, such that for any $0<\delta<\ep$, and for any stationary $n$-varifold $V$ in $(M, g)$ satisfying
\begin{equation}
\label{E: assumption for constancy}
\frac{||V||\big(M \backslash N_{\delta}({S})\big)}{||V||} \leq \delta, 
\end{equation}
we have 
\begin{equation}
\label{E: estimates for varifold distance}
\mathbf{F}(\frac{[S]}{||S||},\frac{V}{||V||}) \leq C(g, S)\sqrt{\delta}.
\end{equation}
Moreover the constant $C(g, S)$ can be chosen to be locally bounded in the set of pairs $(g, S)$ endowed with the product $C^3$ topology.

\end{theorem}

\begin{remark}
\begin{itemize}[leftmargin=1cm]
\item
Suppose that $S$ is a $2$-sided, closed, embedded, strictly stable, minimal hypersurface of $(M,g)$. We have previously seen that for a {$C^q$} metric $\hat{g}$ in a small $C^q$-neighborhood $\mathcal{U}$ of $g$ ($q \geq 3$), there is a unique stable minimal hypersurface $\hat{S}\subset (M,\hat{g})$ coming from $S$ and $C^3$ close to $S$. Thus there is an even smaller $C^q$-neighborhood $\mathcal{U}'\subset \mathcal{U} $ of $g$, such that when the theorem above is applied to the pairs $(\hat{g},\hat{S})$ where $\hat{g}\in \mathcal{U}'$, the constant $C(\hat{g},\hat{S})$ can be chosen uniformly, i.e. independently of $\hat{g}\in \mathcal{U}'$.
\item {In the theorem, the bound in $\sqrt{\delta}$ is probably not sharp as $\delta$ gets small.}
\end{itemize}
\end{remark}

\begin{proof}
To simplify the notations, we can assume $\|V\|(M)=1$. Let $\mu_V$ be the Radon measure on $M$ associated with $V$, and let us write
\[ d\mu_S=\frac{1}{\|S\|}dS. \]
We use $G_n(U)$ to denote the $G(n+1, n)$-Grassmannian bundle of unoriented $n$-planes over a subset $U\subset M$. 

We first prove that, under the assumptions of the theorem,  $V$ is supported over $n$-planes that are ``almost parallel" to those of the foliations $\{S_t\}$. Recall that $\nu$ is the vector field $\nabla t$. Given an $n$-plane $P\in G_n(M)$ in the tangent space $T_xM$, denote by $\nu_P$ a choice of the unit normal of $P$; (note that the following statement does not depend on the choice of $\nu_P$). 

\begin{lemma1}
There exists a constant $C(g, S)$ depending only on $g$ and $S$, such that for any $0<\delta<\epsilon$,
\[ \int_{G_n( N_{\delta}(S))}(1-(\nu(x)\cdot \nu_P)^2)dV(x, P) <C(g,S)\delta. \]
\end{lemma1}
\begin{proof}
We can extend the signed distance function to $M$ with a uniform $C^2$-bound:
\[ t: M\to \R,\quad \text{ with } ||t||_{C^2(M)}\leq C(g,S). \]
Consider the vector field $X=t\nabla t$, and plug $X$ into the first variation formula for $V$:
\[
\begin{aligned}
0 & = \int_{G_n(M)} \Div_P X(x) dV(x, P) \\
   & = \int_{G_n(N_{\delta}(S))} + \int_{G_n(M\setminus N_{\delta}(S))} \Div_P (t\nabla t) dV(x, P) \\
   & = \int_{G_n(N_{\delta}(S))} |\pi_P(\nu(x))|^2 dV(x, P) + \int_{G_n(N_{\delta}(S))} t\, \Div_P(\nabla t) dV(x, P)\\
   & + \int_{G_n(M\setminus N_{\delta}(S))} \Div_P (t\nabla t) dV(x, P).
\end{aligned}
\]
Here $\pi_P: T_xM\to P$ is the orthogonal projection map. And we naturally have
\[ |\pi_P(\nu(x))|^2 = 1-(\nu(x)\cdot \nu_P)^2. \] 
The conclusion then follows directly from the assumption (\ref{E: assumption for constancy}), the fact that $|t|\leq \delta$ on $N_{\delta}$(S), and the $C^2$-bound for $t$.
\end{proof}

\vspace{1em}
We now continue the proof of Theorem \ref{T:constancy}. 
For $x\in N_\ep(S)$, denote by $S_x$ the leaf in the foliation $\{S_t\}$ containing $x$. Since $\nu(x)$ is the unit normal of $S_x$, we get for a constant $C>0$:
\[ 1- (\nu(x)\cdot \nu_P)^2 = \sin^2 \angle(\nu(x), \nu_P) \geq \frac{1}{C} \dist_g(P, T_xS_x)^2. \]
Here $\dist_g(\cdot, \cdot)$ is the choice of distance function on the Grassmannian manifold $G(n+1, n)$ we explained before the statement of the theorem.
Hence by the Cauchy-Schwartz inequality we have
\begin{equation}
\label{E: first estimate}
\int_{G_n(N_{\delta}(S))} \dist_g(P, T_xS_x) dV(x, P) < C(g, S)\sqrt{\delta}. 
\end{equation}

By the definition of the varifold distance $\mathbf{F}$, to prove (\ref{E: estimates for varifold distance}), we need to estimate the following quantity
\[ \Big| \int_{G_n(M)} f(x, P) dV(x, P) - \int_S f(p, T_pS)d\mu_S(p)\Big| \]
for any Lipschitz function $f$ defined on the $G(n+1, n)$-Grassmannian bundle of $M$ (endowed with $\dist_g(\cdot, \cdot)$) with $\|f\|_\infty\leq 1$ and Lipschitz constant $\operatorname{Lip}(f)\leq 1$.

For any point $p\in S$, write $f_S(p)=f(p, T_pS)$ and denote by $f_{av}=\int_S f_S d\mu_S$ the average of $f_S$ over $S$, then
\[ |f_{av}|\leq 1. \] 
We can solve the following PDE on $S$
\begin{equation}
\label{E: laplace equation for f}
f_S= f_{av} + \lap_S \varphi, 
\end{equation}
under the assumption that the average $\int_S \varphi d\mu_S=0.$ By standard elliptic estimates, we have
\[ \| \varphi \|_{C^{2, \al}(S)}\leq C(g,S). \]
We can then find an extension of $\varphi$ to the whole manifold $M$ (still denoted by $\varphi$) such that on $N_\ep(S)$,
\[ \varphi(x) = \varphi(\pi(x)), \quad \text{ for } x \in N_\ep(S),  \]
and the $C^2$ norm of $\varphi$ is still bounded by a constant $C(g,S)$.
By construction $\nabla_\nu \varphi=0$ in $N_\epsilon(S)$.

Given any local coordinate system $\{x^i\}_{i=1, \cdots, n}$ on $S$, we can extend it to a local coordinate system $\{x^1, \cdots, x^n, t\}$ on $N_\ep(S)$ by adding the variable $t$. Note that we naturally have $\partial_t \varphi=0$. We will use this coordinate system to estimate the Hessian matrix of $\varphi$. First we have
\[ \big((\Hess \varphi)|_S\big)_{ij} = (\Hess^S \varphi|_S)_{ij}, \quad \text{for } i, j=1,\cdots, n. \]
Given any $x\in N_\ep(S)$, we have that
\[
\begin{aligned}
\Hess \varphi (x) 
& =
\begin{bmatrix}
\varphi_{tt}-\nabla_{\nabla_\nu \nu}\varphi & \varphi_{jt}-\Ga^k_{j t}(x) \varphi_k (x) \\
\varphi_{it}-\Ga^k_{i t}(x) \varphi_k(x) & \varphi_{ij}(x) -\Ga^k_{ij}(x)\varphi_k(x)
\end{bmatrix}\\
& =
\begin{bmatrix}
0 & -\Ga^k_{j t}(x) \varphi_k (\pi(x)) \\
-\Ga^k_{i t}(x) \varphi_k(\pi(x)) & \varphi_{ij}(\pi(x)) -\Ga^k_{ij}(x)\varphi_k(\pi(x))
\end{bmatrix},
\end{aligned}
\]
where we used the Einstein sum convention for $k=1, \cdots, n$.
Therefore, we see that for any $x\in N_{\delta}(S)$
\[  \big|(\Hess \varphi)_{ij} (x) - (\Hess \varphi)_{ij} (\pi(x)) \big| \leq C(g, S)\delta. \]
So for any $x\in N_{\delta}(S)$,
\begin{equation}
\label{E: trace hessian}
\begin{split}
&\quad\, \big| \Tr_{S_x}(\Hess\varphi)(x) - \Tr_{S} (\Hess\varphi)(\pi(x))\big| \\
& = \big| g^{ij}(x) (\Hess \varphi)_{ij} (x) - g^{ij}(\pi(x)) (\Hess\varphi)_{ij}(\pi(x)) \big| \\
& \leq C(g, S)\delta.
\end{split}
\end{equation}
Another observation is that 
\[ \lap_S \varphi = \Tr_S (\Hess^S \varphi) = \Tr_S (\Hess \varphi) \]
where $\Hess^S$ is the Hessian intrinsic to $S$.

Remember that we need to estimate
\[ \Big| \int_{G_n(M)} f(x, P) dV(x, P) - \int_S f(p, T_pS)d\mu_S(p)\Big|. \]
First we clearly have by (\ref{E: assumption for constancy}) and (\ref{E: first estimate})
\[
\begin{aligned}
 \Big| \int_{G_n(M)} & 
 f(x, P) dV(x, P) - \int_S f(p, T_pS)d\mu_S(p)\Big| \\
&\leq \delta+ \Big| \int_{G_n(N_{\delta}(S))} f(x, P) dV(x, P) - \int_S f(p, T_pS)d\mu_S(p) \Big| \\
& \leq C(g,S)\sqrt{\delta}+ \Big| \int_{N_{\delta}(S)} f(x, T_xS_x) d\mu_V(x) - \int_S f(p, T_pS)d\mu_S(p)  \Big|\\
& \leq C(g,S)\sqrt{\delta}+ \Big| \int_{N_{\delta}(S)} f(\pi(x), T_{\pi(x)}S) d\mu_V(x) - \int_S f(p, T_pS)d\mu_S(p)  \Big|.
\end{aligned}
\]

Next, by (\ref{E: laplace equation for f}), (\ref{E: trace hessian}) and (\ref{E: first estimate}), we get
\[
\begin{aligned}
\Big| \int_{N_{\delta}(S)} 
& f(\pi(x), T_{\pi(x)} S) d\mu_V(x)  - \int_S f(p, T_pS)d\mu_S(p)\Big|\\
&\leq\Big| \int_{N_{\delta}(S)} \big(f_{av}+\lap_S\varphi (\pi(x)) \big)d\mu_V(x)  - f_{av}\Big|\\
& \leq  \delta + \Big| \int_{N_{\delta}(S)} \lap_S\varphi (\pi(x)) d\mu_V(x) \Big|\\
& =  \delta + \Big| \int_{N_{\delta}(S)} \Tr_{S} (\Hess \varphi)(\pi(x)) d\mu_V(x) \Big|\\
& \leq C(g, S)\delta + \Big| \int_{N_{\delta}(S)} \Tr_{S_x} (\Hess \varphi)(x) d\mu_V(x) \Big|\\
& \leq C(g, S)\delta + \Big| \int_{G_n(N_{\delta}(S))} \big[ \Tr_{S_x} (\Hess \varphi)(x)- \Tr_{P}(\Hess \varphi)(x) \big] dV(x, P) \Big|\\
& \quad + \Big| \int_{G_n(N_{\delta}(S))} \Tr_P (\Hess \varphi)(x) dV(x, P) \Big|\\
& \leq C(g, S)\delta +C( \|\varphi\|_{C^2})  \Big| \int_{G_n(N_{\delta}(S))} \dist_g(T_xS_x, P) dV(x, P) \Big| \\
& \quad + \Big| \int_{G_n(N_{\delta}(S))} \Tr_P (\Hess \varphi)(x) dV(x, P) \Big|\\
& \leq C(g, S)\sqrt{\delta} + \Big| \int_{G_n(N_{\delta}(S))} \Div_P (\nabla\varphi)(x) dV(x, P) \Big| \\
& \leq C(g, S)\sqrt{\delta} + \Big| \int_{G_n(M) }\Div_P (\nabla\varphi)(x) dV(x, P) \Big| \leq  C(g, S)\sqrt{\delta}.
\end{aligned}
\]
As usual the constant $C(g,S)$ can change from line to line. Note that in the last line,  we used the fact that $V$ is stationary, i.e. $\int_{G_n(M)} \Div_P X(x) dV(x, P)=0$ for any $C^1$ vector field $X$ on $M$.

Finally all these estimates imply 
\[\Big| \int_{G_n(M)} f(x, P) dV(x, P) - \int_S f(p, T_pS)d\mu_S(p)\Big| \leq C(g,S)\sqrt{\delta},\]
 and it is clear from our computations that the constant $C(g,S)$ can be chosen to be uniformly bounded in a $C^3$-neighborhood of $(g,S)$.
 \end{proof}

\section{Metric deformation and approximation by minimal hypersurfaces}
\label{S:metric deformation}

Let $(M^{n+1},g)$ be a connected closed $(n+1)$-dimensional manifold with $3\leq n+1\leq 7$. The main result of this section is a deformation result saying that, given a connected, closed, embedded, strictly stable, $2$-sided, minimal hypersurface $S$, by perturbing the metric slightly one can construct a minimal hypersurface of arbitrarily large area and Morse index, approximating $S$ after renormalization.

We first start by recalling some facts related to White's Structure Theorem for the space of minimal submanifolds \cite[Theorem 2.1]{White91}. Take $q \geq 7$ from now on.
Recall that a closed embedded minimal hypersurface $\Gamma\subset (M,g)$ is non-degenerate if it has no nontrivial Jacobi fields. A non-degenerate minimal hypersurface is isolated in any $C^{2,\alpha}$ topology. By \cite{White91}, if $\Gamma\subset (M,g)$  is non-degenerate, then for any metric $g'$ close enough to $g$ in the $C^q$ topology there is a unique minimal hypersurface in $(M,g')$ close to $\Gamma$ in the $C^{j,\alpha}$ topology ($j\leq q-1$) and it depends in a $C^{q-j}$ way on $g'$. 
We will say that $\Gamma\subset (M,g)$ deforms to another minimal hypersurface in $(M,g')$, and by abuse of notations, we will usually still denote this new hypersurface (minimal with respect to $g'$) by $\Gamma$.

Recall that $\mathbf{\Gamma}^{(q)}$ denotes the set of $C^q$ metrics on $M$. Let $\mathfrak{M}$ be the set of pairs $(\gamma, [u])$ as in \cite{White91}, where $\gamma\in \mathbf{\Gamma}^{(q)}$ and $u$ is a $C^{q-2,\alpha}$ minimal embedding of an $n$-dimensional connected closed manifold $\Gamma$ inside $(M,\gamma)$. Note that a $C^{2,\alpha}$ minimal embbeding in a $C^q$ metric is $C^{q,\alpha}$ by elliptic regularity; (see for instance \cite[Theorem 1.1, (4)]{White91}). By \cite[Theorem 2.1]{White91}, this space has the structure of a separable $C^{2}$ Banach manifold. The natural projection $\Pi:\mathfrak{M} \to \mathbf{\Gamma}^{(q)}$ is a $C^{2}$ Fredholm map with Fredholm index $0$. A pair $(\gamma,[u])$ is a critical point of $\Pi$ if and only if the minimal embedding $u$ into $(M,\gamma)$ has a nontrivial Jacobi field. Hence a regular value of the projection $\Pi$ is an embedded bumpy metric. In this paper we need to carefully distinguish embedded bumpy metrics from immersed bumpy metrics: while the above general discussion about transversality characterizes regular values of $\Pi$ as the embedded bumpy metrics, it is possible that such metrics are not immersed bumpy. Min-max theorems which produce minimal hypersurfaces with Morse index lower bounds generally need immersed bumpiness, not just embedded bumpiness. The following paragraphs are a preparation to deal with this issue.

Denote by $R_{\mathfrak M}$ and $C_{\mathfrak M}\subset \mathfrak M$ respectively the sets of regular and critical points of $\Pi$. Given a regular pair $(\gamma, [u])\in R_{\mathfrak M}$, if $[u]$ is $1$-sided, its connected $2$-sided double cover may still carry nontrivial Jacobi fields. In this paragraph, we describe the structure of the set of such $1$-sided pairs, whose $2$-sided double covers carry a positive Jacobi field. Set 
\[
\begin{aligned}
\mathfrak S_{\textrm{1-sided}} := 
&\{(\gamma, [u])\in R_{\mathfrak M}: \textrm{$[u]$ is $1$-sided, but its $2$-sided}\\
&\textrm{  double cover carries a positive Jacobi field}\}.  
\end{aligned}
\] 
\begin{lemma}
The set $\mathfrak S_{\textrm{1-sided}}$  is a $C^2$ Banach submanifold of $\mathfrak M$, of codimension $1$. 
\end{lemma}
\begin{proof} 
By \cite{White91}, we can cover $R_{\mathfrak M}$ by countably many open sets $\{\mathcal{U}_i\}_{i\in\N}$ so that $\Pi$ maps each $\mathcal{U}_i$ diffeomorphically onto its image $\Pi(\mathcal{U}_i)\subset \mathbf{\Gamma}^{(q)}$. Therefore we only need to consider $\mathfrak S_{\textrm{1-sided}}\cap \mathcal{U}_i$, and we will omit the subindex $i$. Given $(\gamma, [u])\in \mathfrak S_{\textrm{1-sided}}\cap \mathcal{U}$, where $u:\Gamma\to M$, we know that in a neighborhood $\mathcal{V}\subset\mathcal{U}$ of $(\gamma, [u])$, all other pairs $(\gamma', [u'])$ come from graphical deformations from $(\gamma, [u])$ and hence are $1$-sided. Consider the following functional defined in this neighborhood $\mathcal{V}$
\[ \widetilde\lambda_1: \gamma' \to \widetilde\lambda_1(\gamma',{[\widetilde{u'}]}), \]
where $\widetilde\lambda_1(\gamma',{[\widetilde{u'}]})$ is the first eigenvalue of the Jacobi operator of the $2$-sided double cover $\widetilde{u'} : \widetilde{\Gamma}\to M$ of $u'$. Note that $\mathfrak{M}$ and $\Pi$ are $C^{2}$ by \cite{White91}, and the $C^{q-2}$ embedding $u'$ depends in a $C^{2}$ manner with respect to $\gamma'\in \mathbf{\Gamma}^{(q)}$. Hence the coefficients of that Jacobi operator, viewed as elements of $C^{q-4}(\tilde{\Gamma})$, also depend in a $C^{2}$ manner with respect to $\gamma'\in \mathbf{\Gamma}^{(q)}$. By Lemma \ref{L:Differentiability of the first eigenvalue} in the Appendix, $\widetilde\lambda_1: \Pi(\mathcal V)\subset\mathbf{\Gamma}^{(q)}\to \R$ is a {$C^{2}$} map. 

\vspace{1em}
\textbf{Claim:} \textit{The differential of $\widetilde\lambda_1$ is nonzero at any metric in $\Pi(\mathcal V)\subset\mathbf{\Gamma}^{(q)}$.}

\begin{proof}[Proof of Claim]
It is essentially a computation appearing in the proof of \cite[Lemma 2.6]{White17}.  Given $(\gamma, [u])\in \mathcal V$, $\Gamma$ and $\widetilde{\Gamma}$, consider the following metric perturbation.  Choose a small open neighborhood $U\subset M$ such that $u(\Gamma) \cap U$ is a $2$-sided $n$-ball with a choice of local unit normal $\nu$. Since by elliptic regularity $u$ is a $C^q$ embedding, $\nu$ is a $C^{q-1}$ map. Let $f\in C^{q-1}_c(U)$ be a compactly supported function such that
\[ f=0, Df=0 \text{ and } \Hess f(\nu, \nu)\geq 0 \text{ on } U\cap u(\Gamma), \] 
and
\[ \Hess f(\nu, \nu) >0 \text{ on some proper open subset of } U\cap u(\Gamma). \]
Consider the $1$-parameter smooth perturbation $g_t = e^{2tf} g$ in $\mathbf{\Gamma}^{(q-1)}$. The discussions preceding the Claim still hold for $q-1$ replacing $q$. Note that a $C^q$ metric which is a regular value of $\Pi:\mathfrak{M}\to \mathbf{\Gamma}^{(q-1)}$ is also a regular value of $\Pi:\mathfrak{M}\to \mathbf{\Gamma}^{(q)}$ and vice versa.
One can check that $u(\Gamma)$ remains minimal under $g_t$ and the restrictions $(g_t)|_{u(\Gamma)}$ remain unchanged. Denote by $\varphi_{1, t}$ a normalized first eigenfunction of the connected $2$-sided double cover $\tilde{u}(\widetilde\Gamma)$ of $u(\Gamma)$ under $g_t$, which by Lemma \ref{L:Differentiability of the first eigenvalue} of the Appendix, can be chosen to be {$C^{1}$} in $t$, then
\[ 
\begin{aligned}
\widetilde\lambda_1(g_t, [\tilde u]) 
&= \int_{\tilde u(\widetilde\Gamma)} |\nabla \varphi_{1, t}|^2 - (\Ric_t(\nu, \nu)+|A_t|^2) \varphi_{1, t}^2\\
&= \int_{\tilde u(\widetilde\Gamma)} |\nabla \varphi_{1, t}|^2 - (\Ric(\nu, \nu)+|A|^2) \varphi_{1, t}^2 \\
& + t \int_{\tilde u(\widetilde\Gamma)} (n-1)\Hess f(\nu, \nu) \varphi_{1, t}^2.\\
\end{aligned} 
\] 
The first integral above has vanishing $t$-derivative at $t=0$ since $\varphi_{1,0}$ is a first eigenfunction of the Jacobi operator at $t=0$. The second integral has positive $t$-derivative at $t=0$ by assumption on $f$ and because $\varphi_{1,0}$ is nowhere vanishing. Now approximate the variation $g_t$ by a smooth variation $\tilde{g}_t$ of $C^q$ metrics, so that the $t$-derivative of $\widetilde\lambda_1(\tilde{g}_t, [\tilde{u}_t])$ at $t=0$ is still positive. That proves the claim. 
\end{proof}

The Implicit Function Theorem implies that the intersection $\mathfrak S_{\textrm{1-sided}}\cap \mathcal{V}$ is a codimension $1$, $C^2$ Banach submanifold of $\mathcal{V}$. Since there are only countably many such $1$-sided pairs for each fixed $\gamma$, we finish the proof.
\end{proof}

Let $g_t:[0,1]\to \mathbf{\Gamma}^{(q)}$ be a smooth $1$-parameter family of metrics on $M$. By Smale's transversality theorem \cite{Smale65}, since the regularity of $\Pi$ is $C^2$, we can perturb $g_t$ slightly in the $C^\infty$ topology to get another smooth $1$-parameter family $\tilde{g}_t:[0,1]\to \mathbf{\Gamma}^{(q)}$ that is transversal to both $\Pi:\mathfrak{M} \to \mathbf{\Gamma}^{(q)}$ and $\Pi\big|_{\mathfrak S_{\textrm{1-sided}}}: \mathfrak S_{\textrm{1-sided}} \to \mathbf{\Gamma}^{(q)}$. In particular $\Pi^{-1}(\{\tilde{g}_t\}) \subset \mathfrak{M}$ is a smooth $1$-dimensional submanifold with boundary of $\mathfrak{M}$. By Sard theorem, for almost every $t\in [0,1]$ in the sense of Lebesgue measure, $\tilde{g}_t$ is a regular value of the restriction $\Pi\big|_{\Pi^{-1}(\{\tilde{g}_t\})} : \Pi^{-1}(\{\tilde{g}_t\}) \to \mathbf{\Gamma}^{(q)}$ and thus a regular value of $\Pi : \mathfrak{M} \to \mathbf{\Gamma}^{(q)}$ by transversality of $\{\tilde{g}_t\}$. In other words, for almost every $t\in [0,1]$, $\tilde{g}_t$ is an embedded bumpy metric. Moreover, by transversality $\Pi^{-1}(\{\tilde{g}_t\})$ intersects $\mathfrak S_{\textrm{1-sided}}$ locally only at finitely many points, and so for almost every $t\in [0,1]$ in the sense of Lebesgue measure, for any $1$-sided minimal embedding $[u]$ under $\tilde{g}_t$, its $2$-sided double cover does not have positive Jacobi fields. Thus these metrics satisfy a condition stronger than embedded bumpiness but a priori weaker than immersed bumpiness, nevertheless we have seen in the proof of Lemma \ref{L:derivative of cylindrical widths} that this condition was enough to get multiplicity one and lower Morse index bounds in min-max constructions.

In the following deformation theorem, the norms $||.||_{C^m}$ are computed with a fixed metric $g$.

\begin{theorem}
\label{T:deform}
Let $(M,g)$ be a closed Riemannian manifold of dimension $3\leq n+1\leq 7$ endowed with a smooth metric $g$. For any integer $m>0$ and for any connected, closed, embedded, minimal hypersurface $S\subset (M,g)$ which is $2$-sided and strictly stable, there is a smooth metric  $g'$ with
\[ ||g'-g||_{C^m} \leq \frac{1}{m}, \]
so that $S$ deforms to a $2$-sided, strictly stable, connected, closed, embedded, minimal hypersurface, still denoted by $S$, in $(M,g')$, and there is a non-degenerate, closed, embedded, minimal hypersurface $\Sigma \subset (M,g')$ satisfying the following with respect to the metric $g'$: 
\vspace{0.5em}
\begin{enumerate}[label=$(\roman*)$]
\addtolength{\itemsep}{0.7em}
\item 
\hfill
$\displaystyle \Sigma \cap S =\varnothing$,
\hfill\refstepcounter{equation}

\item 
\hfill
$\displaystyle m < ||\Sigma||$,
\hfill\refstepcounter{equation}

\item 
\hfill
$\displaystyle \big| \ind(\Sigma)\cdot ||\Sigma||^{-1} - ||S||^{-1}\big| < \frac{1}{m}$,
\hfill\refstepcounter{equation}

\item 
\hfill
$\displaystyle \mathbf{F}(\frac{[S]}{||S||}, \frac{[\Sigma]}{||\Sigma||}) < \frac{1}{\log(||\Sigma||)}$.
\hfill\refstepcounter{equation}

\end{enumerate}
\vspace{0.5em}
\end{theorem}

\begin{proof}
We are given $(M^{n+1},g)$ ($3\leq n+1\leq 7$), $m>0$ and a minimal hypersurface $S$ as in the statement. For any metric $\hat{g}$ in a small $C^{3}$-neighborhood of $g$, the $2$-sided strictly stable minimal hypersurface coming from $S\subset (M, g)$ will be denoted by $S_{\hat{g}}$.

For any metric $\hat{g}$ in this neighborhood, we cut $(M,\hat{g})$ along $S_{\hat{g}}$ and, after continuously choosing a connected component, we get a compact manifold $(\hat{M},\hat{g})$; (more rigorously, $\hat{M}$ is the metric completion of the chosen component of $(M\backslash S_{\hat{g}},\hat{g})$ and the lift of $\hat{g}$ to $\hat{M}$ is still denoted by $\hat{g}$).  In Section \ref{S:core manifold}, we described how to use min-max theory to define a sequence of min-max widths $\tilde{\omega}_p$ and construct a sequence of closed minimal hypersurfaces $\Gamma_p $ embedded in the interior of $(\hat{M},\hat{g})$, satisfying some properties related to their areas and Morse indices.

In the following paragraph, we construct certain perturbations of the given metric $g$. Let $k$ be a large integer to be fixed later. Consider a smooth nonnegative symmetric $2$-tensor $h := \varphi g$ on $M$, where $\varphi: {M} \to [0,1]$ is a smooth cut-off function supported in a $1/(2k)$-neighborhood of $S$, such that $\varphi=1$ in a tubular neighborhood of $S$.  It is not hard to check that we can choose $k$ arbitrarily large and a corresponding cut-off function $\varphi$ so that 
\begin{equation} \label{kk}
|\nabla^{m'} h| \leq C_{m,S,g}  k^{m'} \text{ for all }  0\leq m'\leq m
\end{equation}
where $C_{m,S,g}$ is a constant depending on $m$, $S$ and $(M,g)$.
Now let 
$$t_k:= \exp(-\delta \sqrt{k})$$ 
where $\delta\in(0,1)$ will be chosen later, and consider the $1$-parameter family of smooth metrics on $M$:
$$g_t := g+th, \quad t\in [0,t_k].$$
By (\ref{kk}) and the choice of $t_k$, any such $g_t$ satisfies
$$||g_t -g||_{C^m} < \frac{1}{m}$$ for $k$ large enough, the norm $||.||_{C^m}$ being measured with $g$. 
In general, $g_t$ may not be bumpy for most $t\in[0,t_k]$, so we need perturb this family in order to obtain minimal hypersurfaces with lower index bound. Take $q> 7+m$. The discussion before the theorem gives an arbitrarily small perturbation of $\{g_t\}$ (viewed as $C^q$ metrics) in the $C^\infty$ topology into another family $\{\tilde{g}_t\}$ of $C^q$ metrics, such that there is a full Lebesgue measure set $\mathcal{A} \subset [0,t_k]$ satisfying the following: for any $t\in \mathcal{A}$, $\tilde{g}_t$ is ``nice'' in the sense that it is embedded bumpy (no closed embedded minimal hypersurface has a nontrivial Jacobi field) and the $2$-sided double cover of any $1$-sided closed embedded minimal hypersurface has no positive Jacobi field (i.e. it is not weakly stable). These metrics $\tilde{g}_t$ on $M$ lift to metrics on $\hat{M}$ still denoted by $\tilde{g}_t$.  We can also ensure that for a constant $\mathbf{c}$ depending only on $(M, g)$, $S$ the following holds: for $k$ large enough, for all $t\in [0,t_k]$, and all $n$-plane $P$ in the Grassmannian of tangent $n$-planes of $\hat{M}$,
\begin{equation} 
\label{gribouilli}
\begin{split}
& ||\tilde{g}_t -g||_{C^m} < \frac{1}{m}, \\
& \big|\Tr_{P, \tilde{g}_t} \frac{\partial \tilde{g}_t}{\partial t} - n\varphi \big| \leq \big|\Tr_{P, \tilde{g}_t} \frac{\partial \tilde{g}_t}{\partial t}  - \Tr_{P, {g}_t} \frac{\partial g_t}{\partial t} \big| +\big| \Tr_{P, g_t} \frac{\partial g_t}{\partial t}  - n\varphi \big|\\
& \quad\quad\quad\quad\quad\quad\quad\,\,\, \leq \exp(-\delta \sqrt{k}) + \big| \Tr_{P, g_t}(\varphi g) - n\varphi \big|\\
& \quad\quad\quad\quad\quad\quad\quad\,\,\, = \exp(-\delta \sqrt{k}) + \big| \frac{n\varphi}{1+t\varphi} - n\varphi \big|\\
& \quad\quad\quad\quad\quad\quad\quad\,\,\, \leq \mathbf{c} \exp(-\delta \sqrt{k}),\\
& \Tr_{P, \tilde{g}_t} \frac{\partial \tilde{g}_t}{\partial t} \leq n+ \mathbf{c} \exp(-\delta \sqrt{k}).\\
\end{split}
\end{equation}

The areas of $S_{{g}_t}$ and $S_{\tilde{g}_t}$ are differentiable in $t$, and by a computation, $S_{{g}_t} = S$ remains minimal as $t$ varies; moreover 
\[ \frac{d}{d t} \Area_{{g}_t}(S_{{g}_t} )= \frac{n}{2} (1+t)^{\frac{n-2}{2}} \Area_{{g}}(S) = \frac{n}{2(1+t)}\Area_{g_t}(S_{g_t}).\] 
For any $k$, by choosing the perturbation $\{\tilde{g}_t\}$ of $\{g_t\}$ smaller if necessary, we can then assume that for all $t\in [0,t_k]$,
\begin{equation} \label{kangaroo1}
\begin{split}
\big|  \frac{d}{d t} \Area_{\tilde{g}_t}(S_{\tilde{g}_t}) - \frac{n}{2} \Area_{\tilde{g}_t}(S_{\tilde{g}_t}) \big| & \leq \mathbf{c}_0\frac{n t_k}{2(1+t_k)}\\
&\leq  \mathbf{c}_0 \exp(-\delta \sqrt{k}),
\end{split}
\end{equation}
where $\mathbf{c}_0$ is a constant depending only on $(M,g)$ and $S$.

We are now ready to describe a quantitative perturbation argument in the same spirit as \cite{Marques-Neves-Song19}, using the cylindrical Weyl Law with remainder term. Consider the cylindrical widths $\{\tilde\omega_p(\hat M, \tilde{g}_t)\}$ of the connected compact manifold with stable boundary $(\hat M, \tilde{g}_t)$, as defined in Section \ref{S:core manifold}. Roughly speaking, taking derivatives of $\tilde\omega_p(\hat M, \tilde{g}_t)$ for large enough $p$ will give rise to the desired minimal hypersurfaces. Recall that the cylindrical Weyl Law Theorem \ref{T:cylindrical Weyl Law} implies for all $p\in\N$:
\begin{equation} 
\label{cylindricalWeyl}
\Area_{\tilde{g}_t}(S_{\tilde{g}_t}) \leq \frac{1 }{p} \tilde{\omega}_p(\hat{M},\tilde{g}_t) \leq \Area_{\tilde{g}_t}(S_{\tilde{g}_t})+ Cp^{-{\frac{n}{n+1}}} \quad\text{ for all $t\in [0,t_k]$}
\end{equation}
for a constant $C>0$ depending on $(M,g)$, $S$, but independent of $t\in [0,t_k]$.  Define
\[ p_k:= [ \exp(2\delta \sqrt{k}) ] +1. \]
The derivative of $\tilde{\omega}_{{p_k}}(\hat{M},\tilde{g}_t)$ exists at almost every $s \in [0,t_k]$ in the sense of Lebesgue measure, because it is a Lipschitz function by Lemma \ref{L:cylindrical withs are Lipschitz}. Let $\mathcal{A}' \subset [0,1]$ be the full Lebesgue measure set of times in $\mathcal{A}$ where the derivative of $\tilde{\omega}_{{p_k}}(\hat{M},\tilde{g}_t)$ exists. Recall that $\mathcal{A}$ is the set of times where the metric $\tilde{g}_t$ is ``nice''. By Lemma \ref{L:derivative of cylindrical widths}, at each $s\in\mathcal{A}'$, there is a closed embedded minimal hypersurface $\Gamma_{p_k}\subset\interior(\hat{M}, \tilde{g}_s)$ (which is necessarily non-degenerate) such that 
\begin{equation} \label{cookie}
\begin{split}
& \displaystyle \tilde{\omega}_{{p_k}}(\hat{M},\tilde{g}_s) =  \Area(\Gamma_{p_k}), \quad \ind(\Gamma_{p_k}) = p_k\quad  \text {and}\\
& \displaystyle \frac{d}{d t}\Big|_{t=s}\tilde{\omega}_{{p_k}}(\hat{M},\tilde{g}_t) =  \int_{\Gamma_{p_k}}\frac{1}{2} \Tr_{\Gamma_{p_k}, \tilde{g}_s}  \big(\frac{\partial \tilde{g}_t}{\partial t}\Big|_{t=s} \big)d\Gamma_{p_k}.
\end{split}
\end{equation}

Consequently at  a  time $s\in\mathcal{A}'$, since $ \Tr_{\Gamma_{p_k}, \tilde{g}_s} (\frac{\partial \tilde{g}_t}{\partial t}|_{t=s}) \leq n+\mathbf{c} \exp(-\delta \sqrt{k})$ by (\ref{gribouilli}), we get
\[ \frac{d}{d t}\Big|_{t=s} \frac{1 }{{p_k}} \tilde{\omega}_{{p_k}}(\hat{M},\tilde{g}_t) \leq (\frac{n}{2} + \mathbf{c} \exp(-\delta \sqrt{k}))\frac{1 }{{p_k}} \tilde{\omega}_{{p_k}}(\hat{M},\tilde{g}_s). \]
Next, by the cylindrical Weyl law (\ref{cylindricalWeyl}) and (\ref{kangaroo1}), we obtain for each $k$ and all $s\in \mathcal{A}'$:
\begin{align}  \label{kangaroo2}
\begin{split}
\frac{d}{d t}\Big|_{t=s} \frac{1}{{p_k}} \tilde{\omega}_{{p_k}}(\hat{M},\tilde{g}_t)     
& \leq (\frac{n}{2}+ \mathbf{c} \exp(-\delta \sqrt{k}))  \frac{1}{{p_k}} \tilde{\omega}_{{p_k}}(\hat{M},\tilde{g}_s)   \\
& \leq \frac{n}{2} \Area_{\tilde{g}_s}(S_{\tilde{g}_s})  + \mathbf{c} \exp(-\delta \sqrt{k}) \Area_{\tilde{g}_s}(S_{\tilde{g}_s})\\
& + (\frac{n}{2}+ \mathbf{c} \exp(-\delta \sqrt{k})) C {p_k}^{-{\frac{n}{n+1}}}\\
& \leq \frac{n}{2} \Area_{\tilde{g}_s}(S_{\tilde{g}_s})  + \mathbf{c}_1  \exp(-\delta \sqrt{k}) + \mathbf{c}_1 {p_k}^{-{\frac{n}{n+1}}}\\
& \leq \frac{d}{d t} \Big|_{t=s} \Area_{\tilde{g}_t}(S_{\tilde{g}_t}) +\mathbf{c}_1\exp(-\delta \sqrt{k}) + \mathbf{c}_1 {p_k}^{-{\frac{n}{n+1}}},
\end{split}
\end{align}
where $\mathbf{c}_1$ depends only on $(M,g)$, $S$ and can change from line to line.

Since we defined $p_k:= [\exp(2\delta \sqrt{k})]+1$, we have $p_k^{-\frac{n}{n+1}} = o(t_k)$. Thus for $k$ large enough (depending on $\delta$, $\mathbf{c}_0$, $\mathbf{c}_1$), by the cylindrical Weyl law (\ref{cylindricalWeyl}) and (\ref{kangaroo2}), there exists $s\in \mathcal{A}'$ such that: 

\begin{equation} \label{boundderivative}
-\frac{1}{k^2}\leq \frac{d}{d t}\Big|_{t=s} \frac{1 }{p_k} \tilde{\omega}_{p_k}(\hat{M},\tilde{g}_t) -  \frac{d}{d t}\Big|_{t=s}\Area_{\tilde{g}_t}(S_{\tilde{g}_t})\leq \frac{1}{k^2}.
\end{equation}
Here the upper bound is ensured by (\ref{kangaroo2}), and then the lower bound at some $s$ follows from (\ref{cylindricalWeyl}) and the Fundamental Theorem of Calculus.

By (\ref{cookie}) we find  a closed non-degenerate minimal hypersurface $\Sigma$ embedded in the interior of $(\hat{M},\tilde{g}_s)$, which projects to an embedded minimal hypersurface $(M,\tilde{g}_s)$ still denoted by $\Sigma$, with support disjoint from $S_{\tilde{g}_s}$. This minimal hypersurface $\Sigma$ has multiplicity one, and index exactly $p_k$. By (\ref{cylindricalWeyl}) its mass satisfies for $k$ large enough 
\begin{equation} 
\label{reminder}
||\Sigma||= \tilde{\omega}_{p_k}(\hat{M},\tilde{g}_s) >m\quad \text{and}\quad  \big|p_k\cdot ||\Sigma||^{-1} - ||S_{\tilde{g}_s}||^{-1}\big| < \frac{1}{m}.
\end{equation}
Furthermore by (\ref{cookie}), (\ref{boundderivative}), (\ref{gribouilli}), (\ref{kangaroo1}) and (\ref{cylindricalWeyl}), we have
\begin{align*}
\frac{1}{k^2} 
& \geq \Big| \frac{1}{p_k} \int_\Sigma \frac{1}{2} \Tr_{\Sigma, \tilde{g}_s} \big(\frac{\partial \tilde{g}_t}{\partial t}\Big|_{t=s}\big) d\Sigma - \frac{d}{d t}\Big|_{t=s} \Area_{\tilde{g}_t} (S_{\tilde{g}_t})\Big| \\
& \geq \Big|  \frac{\tilde{\omega}_{p_k}}{p_k}  \frac{1}{||\Sigma||}\int_\Sigma \frac{n\varphi}{2}  d\Sigma -  \frac{d}{d t}\Big|_{t=s} \Area_{\tilde{g}_t} (S_{\tilde{g}_t}) \Big| - \mathbf{c}_2 \exp(-\delta \sqrt{k}) \\
& \geq \Big| \frac{\tilde{\omega}_{p_k}}{p_k}  \frac{1}{||\Sigma||}\int_\Sigma \frac{n\varphi}{2}  d\Sigma  - \frac{n}{2}\Area_{\tilde{g}_s} (S_{\tilde{g}_s})\Big|  - \mathbf{c}_2 \exp(-\delta \sqrt{k})\\
& \geq \frac{n}{2} \Area_{\tilde{g}_s} (S_{\tilde{g}_s}) \Big|\frac{1}{||\Sigma||}\int_\Sigma \varphi  d\Sigma - 1 \Big| - \mathbf{c}_2 \exp(-\delta \sqrt{k}) -  \mathbf{c}_2 \exp(-\frac{2\delta \sqrt{k} n}{n+1}),
\end{align*}
where $\mathbf{c}_2$ is only depending on $(M,g)$, $S$ and can change from line to line. From that we get for $k$ large enough:
\begin{equation}\label{magentta}
\big| \frac{1}{||\Sigma||}\int_\Sigma \varphi d \Sigma -1\big|< \frac{1}{k}.
\end{equation}
Note that for $k$ large enough, $S_{\tilde{g}_s}$ lies in a $1/(2k)$-neighborhood of $S$ and recall that $\varphi$ is also supported in a $1/(2k)$-neighborhood of $S$ by definition. Inequality (\ref{magentta}) then implies that the varifold $\Sigma$ has most of its mass concentrated in a $1/k$-neighborhood of ${S}_{\tilde{g}_s}$: 
\begin{equation} 
\label{mostly inside}
\frac{||\Sigma \backslash N_{1/k}({S_{\tilde{g}_s}})||}{||\Sigma||} \leq \frac{1}{k}.
\end{equation}
By Section \ref{S:constancy} and our Quantitative Constancy Theorem \ref{T:constancy}, (\ref{mostly inside}) implies that for some constant $C(g,S)$ depending on $(M, g)$ and $S$: 
\begin{equation} 
\label{octopus1}
\mathbf{F}(\frac{[S_{\tilde{g}_s}]}{||S_{\tilde{g}_s}||},\frac{[\Sigma]}{||\Sigma||}) \leq \frac{C(g,S)}{\sqrt{k}}.
\end{equation}
 Since 
 $p_k= [\exp(2\delta \sqrt{k})]+1$ by definition, and by (\ref{reminder}), we get for large $k$:
\begin{equation} \label{octopus2}
\frac{C(g,S)}{\sqrt{k}} \leq \frac{2\delta C(g,S)}{\log(p_k)} \leq \frac{3\delta C(g,S)}{\log(||\Sigma||)}.
\end{equation}
Hence we can take $\delta < \frac{1}{3C(g,S)}$ and a corresponding large $k$, then we slightly perturb the $C^q$ metric  $\tilde{g}_s$ to a smooth metric $g'$ and the theorem is proved by (\ref{reminder}), (\ref{octopus1}), (\ref{octopus2}) and non-degeneracy of $\Sigma$.
 \end{proof}

\section{Generic scarring for minimal hypersurfaces along stable hypersurfaces}
\label{S:scarring}
In this section, we prove our main Theorem \ref{T:main}, as well as a generic scarring result for most closed Riemannian $3$-manifolds, Theorem \ref{T:3dimension}.

\subsection{Proof of Theorem \ref{T:main}}
Now we are ready to prove the main theorem using  Theorem \ref{T:deform}.
\begin{proof}[Proof of Theorem \ref{T:main}]

Let $A>0$. 
We say that a smooth metric $g$ on $M$ satisfies Property ($P_A)$ if:
\begin{itemize}
\item[($P_A)$] any connected, closed, embedded, stable, minimal hypersurface $S\subset (M,g)$ with area at most $A$ has a non-degenerate $2$-sided double cover (which may be disconnected). If moreover $S$ is $2$-sided, there is a closed, embedded, non-degenerate, minimal hypersurface $\Sigma$ satisfying:
\vspace{0.5em}
\begin{itemize}
\addtolength{\itemsep}{0.7em}
\item 
\hfill
$\displaystyle \Sigma \cap S =\varnothing$,
\hfill\refstepcounter{equation}

\item 
\hfill
$\displaystyle A < ||\Sigma||$,
\hfill\refstepcounter{equation}

\item 
\hfill
$\displaystyle \big| \ind(\Sigma) \cdot ||\Sigma|| ^{-1} - ||S||^{-1}\big| < \frac{1}{A}$,
\hfill\refstepcounter{equation}

\item 
\hfill
$\displaystyle \mathbf{F}(\frac{[S]}{||S||}, \frac{[\Sigma]}{||\Sigma||}) < \frac{1}{\log(||\Sigma||)}$.
\hfill\refstepcounter{equation}

\end{itemize}
\end{itemize}

The set of metrics satisfying $(P_A)$ is denoted by $\mathcal{M}_A$. We want to show that for all $A>0$,  $\mathcal{M}_A$ is open and dense in the $C^\infty$ topology. Then $\bigcap_{A>0}  \mathcal{M}_A$ would be a $C^\infty$-generic family of metrics on $M$ in the sense of Baire and any metric in this family would satisfy the main theorem.

Fix $A>0$. If $g\in \mathcal{M}_A$ then by definition of $(P_A)$, by Sharp's compactness result \cite{Sharp17} and a standard Jacobi field argument,   there exists a small $\epsilon_0>0$ depending on $g$, such that no connected, closed, embedded, stable minimal hypersurface in $(M, g)$ has area lying inside $(A, A+\epsilon_0)$, and there are only finitely many such stable minimal hypersurfaces in $(M, g)$ with area at most $A+\epsilon_0$. Therefore, given $g\in \mathcal{M}_A$, any stable minimal hypersurface $S'$ under a small perturbed metric $g'$ with $\Area_{g'}(S')\leq A$ must come from a stable minimal hypersurface $S$ under $g$ with $\Area_g(S)\leq A$. Openness of $\mathcal{M}_A$ then follows from the Implicit Function Theorem.

To prove denseness of  $\mathcal{M}_A$, consider an immersed bumpy metric $g$ on $M$; (recall that the set of immersed bumpy metrics is dense by \cite[Theorem 2.2]{White91} and \cite[Theorem 2.1]{White17}). The set of connected, closed, embedded, stable, minimal hypersurfaces in $(M,g)$ with area at most $2A$ is finite. By slightly rescaling the metric, we can make sure that no such stable closed minimal hypersurface has area between $A-\epsilon_0$ and $A+\epsilon_0$ for some small $\epsilon_0>0$. All the deformations below will be chosen small enough to preserve that condition. Let us call those connected, closed, embedded, stable, minimal hypersurfaces with area less than $A$ which are $2$-sided by $\{S_1,...,S_J\}$, and those which are $1$-sided by $\{T_1,...,T_K\}$.  
We can now apply the deformation result Theorem \ref{T:deform} to each $S_j$ successively, and get a metric $g'\in \mathcal{M}_A$ arbitrarily close to $g$. More precisely, fix an $\epsilon>0$ and an integer $l>0$. We use Theorem \ref{T:deform} with $m>l$ very large depending on $A$ and $\epsilon$, on the stable hypersurface $S_1$ and we get a metric $g_1$ such that $||g-g_1||_{C^l}\leq \epsilon/J$, together with a non-degenerate minimal hypersurface $\Sigma_1$ satisfying the four bullets of Theorem \ref{T:deform} with $A$ (resp. $S_1$, $\Sigma_1$) replacing $m$ (resp. $S$, $\Sigma$). We can now apply Theorem \ref{T:deform} to the next stable hypersurface $S_2$ (deformed from $S_2$ and still denoted by $S_2$), and we get another metric $g_2$ with $||g-g_2||_{C^l}\leq 2\epsilon/J$ and a corresponding minimal hypersurface $\Sigma_2$. We can also ensure that $\Sigma_1$ deforms to a minimal hypersurface still denoted by $\Sigma_1$ in $(M, g_2)$ and still satisfying the four bullets of Theorem \ref{T:deform}. We continue to deform the metric for each subsequent $S_i$ ($i=3,...,J$) and eventually we get a metric $g_J$ with  $||g-g_J||_{C^l}\leq \epsilon$ and non-degenerate minimal hypersurfaces $\Sigma_1,...,\Sigma_J$. The stable hypersurfaces $S_1,...,S_J \subset (M,g)$ deform to stable hypersurfaces still denoted by $S_1,...,S_J$. One checks that if the deformation at each step is taken small enough,  the only connected, closed, embedded, stable hypersurfaces with area less than $A$ in $(M,g_J)$ are $S_1,...,S_J$ and the minimal hypersurfaces coming from $T_1,...,T_K$. Thus, $g_J\in \mathcal{M}_A$ and this concludes the proof of the denseness of $\mathcal{M}_A$.
\end{proof}

\subsection{Generic scarring in most closed Riemannian $3$-manifolds}

As a consequence of the proof of Theorem \ref{T:main}, we can show that scarring occurs for a generic metric on any closed $3$-manifold which does not admit a metric of constant positive curvature.

\begin{proof}[Proof of Theorem \ref{T:3dimension}]
Let $M^3$ be a closed $3$-manifold not diffeomorphic a quotient of the $3$-sphere. Then by the Geometrization Conjecture \cite{Perelman03, Kleiner-Lott08, Morgan-Tian14, Bessieres-Besson-Maillot-Boileau-Porti10, Morgan-Fong10} and the Virtually Haken Conjecture \cite[Theorems 9.1 and 9.2]{Agol13}, there is an oriented finite cover $\pi:M_{cover}\to M$, so that  $M_{cover}$ is either not irreducible or contains an incompressible $2$-torus or has non-vanishing second Betti number. Let $S_0 \subset M_{cover}$ be either a topologically non-trivial embedded $2$-sphere, or an {embedded} incompressible $2$-torus or an embedded closed surface representing a non-trivial element in the second homology group of $M_{cover}$, according to the case.

For a smooth metric $g$ on $M$, which lifts to a metric on $M_{cover}$ still denoted by $g$, one can minimize the area of $S_0$ in its isotopy class by Meeks-Simon-Yau \cite{Meeks-Simon-Yau82}, and get a connected, closed, smooth embedded, minimal surface $S_1\subset (M_{cover},g)$. By taking a double cover if necessary, we can assume that $S_1$ is $2$-sided stable.
Moreover, by White's Bumpy Metric Theorem \cite{White91, White17} and Transversality Theorem \cite{White19}, if the metric $g$ is well chosen, $S_1$ is strictly stable and the image $\pi(S_1)$ is a strongly self-transverse immersed surface, which implies that the surface is self-transverse and the multiplicity of the image $\pi(S_1)$ is at most $2$ at all points of $\pi(S_1)$ except maybe a finite number of points where three different tangent planes intersect transversely; (see \cite[Theorem 20, Theorem 21]{White19} for definitions of being strongly self-transverse). Let $\mathcal{M}_{\text{transverse}}$ be the set of metrics $g$ on $M$ such that there is a connected, closed, embedded, 2-sided, strictly stable, minimal surface $S_1\subset (M_{cover},g)$ such that $\pi(S_1)$ is strongly self-transverse. By \cite{White91, White17,White19}, $\mathcal{M}_{\text{transverse}}$ is open and dense.

We would like to apply the quantitative perturbation arguments used in the proof of Theorem \ref{T:main}. The idea is to generically construct minimal surfaces in $(M_{cover},g)$ and then project back to the original manifold $M$. However, in order to be able to project the minimal surfaces downstairs, we need to deform the metric $g$ in an equivariant way so that it projects to an honest metric on $M$. Let us see how we can prove the following analogue of Theorem \ref{T:deform}.

\vspace{1em}
\textbf{Fact:} \textit{Let $\pi:M_{cover} \to M$, $g\in \mathcal{M}_{\text{transverse}}$, and $S_1\subset (M_{cover},g)$ be as above. For any integer $m>0$, there is a {smooth} metric $g'$ with $\|g'-g\|_{C^m}<\frac{1}{m}$, which lift to the metric $g'$ on $M_{cover}$,
such that $\pi(S_1)$ deforms to an immersed minimal surface $S_{g'} \subset (M,g')$, and there is a non-degenerate, connected, closed, embedded, minimal surface $Z \subset (M_{cover},g')$ satisfying the following with respect to $g'$:}
\vspace{0.5em}
\begin{enumerate}[label=$(\roman*)$]
\addtolength{\itemsep}{0.7em}
\item 
\hfill
$\displaystyle m < ||Z||, \quad m< \ind(Z)$,
\hfill\refstepcounter{equation}

\item 
\hfill
$\displaystyle \mathbf{F}(\frac{[S_{g'}]}{||S_{g'}||}, \frac{[\pi(Z)]}{||\pi(Z)||}) < \frac{1}{m}$.
\hfill\refstepcounter{equation}
\end{enumerate}
\vspace{0.5em}
\textit{Here $\pi(Z)$ denotes as usual the image of $Z$ by the projection $\pi$, $||.||$ denotes the area and $[\pi(Z)]$ is the multiplicity one integer rectifiable varifold associated to the image $\pi(Z)$.} 
\vspace{1em}

Once this Fact is checked, one can argue as follows to conclude. Since the subset of metrics $\mathcal{M}_{\text{transverse}}$ is dense, it is enough to prove the theorem in a small neighborhood of any metric $g\in  \mathcal{M}_{\text{transverse}}$. Given such a $g$, let $\pi(S_1)$ be the associated strongly self-transverse minimal surface constructed in the first paragraph. There is a $C^\infty$-neighborhood $\mathcal{U}$ of $g$ so that for any metric $g''\in \mathcal{U}$, the minimal surfaces $\pi(S_1) \subset (M,g)$ and $S_1\subset (M_{cover},g)$ deform respectively uniquely to some strictly stable minimal surfaces $S_{g''}$ and $\tilde{S}_{g''}$, and $S_{g''}$ remains strongly self-transverse (which is an open condition). For any $m>0$, let $\mathcal{M}'_{\mathcal{U},m}$ be the set of metrics $g'\in \mathcal{U}$ such that there is a non-degenerate, connected, closed, embedded, minimal surface $Z\subset  (M_{cover},g')$ with 
$$\Area(Z)>m, \quad \ind(Z)>m,$$
$$\mathbf{F}(\frac{[S_{g'}]}{||S_{g'}||}, \frac{[\pi(Z)]}{||\pi(Z)||}) < \frac{1}{m}.$$
Then $\mathcal{M}'_{\mathcal{U},m}$ is open by the Implicit Function Theorem, and dense in $\mathcal{U}$ by the Fact. By taking the intersection $\bigcap_{m>0} \mathcal{M}'_{\mathcal{U},m}$, one concludes for the neighborhood $\mathcal{U}$ (note that the degree $\deg_\pi$ of the covering map $\pi:M_{cover}\to M$ is fixed so the area of $\pi(Z)$ is at least $m/\deg_\pi$), and thus one finishes the proof.

It remains to verify the Fact. Let $g\in \mathcal{M}_{transverse}$ as before. Similarly to the proof of Theorem \ref{T:deform}, the idea is to perturb the metric $g$ and use the cylindrical Weyl law to find $\Sigma$. However, in the present case, the fact that we work on a cover will cause some issues.  On the other hand, embeddedness of the stable surface was previously used to construct effective deformations, and here we do not require any quantitative estimates on how close $\Sigma$ is to $S_{g'}$ depending on the area of $\Sigma$. Let $h:=\varphi g$ be a nonnegative symmetric $2$-tensor on $M$, where $\varphi:M\to [0,1]$ is a smooth cut-off function supported in a $1/(2k)$-neighborhood of the immersed surface $\pi(S_1)$ , and $\varphi=1$ in a smaller tubular neighborhood of $\pi(S_1)$. This tensor lifts to $\hat{\varphi}g$ on $M_{cover}$ still called $h$ by abuse of notations, where $\hat\varphi$ is equal to $1$ on a neighborhood of $S_1$. $\hat\varphi$ is not supported in a $1/(2k)$-neighborhood of $S_1$, but instead is supported in the $1/(2k)$-neighborhood of the union of the lifts of $\pi(S_1)$ that we call
$$\mathfrak{X} : =\pi^{-1}\big(\pi(S_1)\big).$$ 
The set $\mathfrak{X} $ is a finite union of embedded strictly stable minimal surfaces, and it is strongly self-transverse.
Let $T_1,...,T_L$ be a maximal disjoint family of connected lifts of $\pi(S_1)$ embedded in $M_{cover}$, meaning that any lift of $\pi(S_1)$ has to intersect one of the $T_i$. After cutting $(M_{cover},g)$ along $T_1\cup...\cup T_L$ and choosing a connected component, we are left with a compact manifold $N_{g}$ with strictly stable boundary.
Consider the $1$-parameter family of metrics on $M_{cover}$: 
$$g_t:=g + t{h}.$$ 
We can slightly perturb $g_t$ to a family $\tilde{g}_t$ of $C^q$ metrics transversal with respect to White's projection $\Pi$ (on $M_{cover}$). Note that $\tilde{g}_t$  may not be equivariant and may not descend to a metric on $M$. For any metric $g''$ close enough to $g$, the $T_i$ deform to some nearby surfaces, so $N_{g}$ deforms to a compact manifold $N_{g''}$ with strictly stable boundary.  Similarly $\mathfrak{X}$ deforms to $\mathfrak{X}_{g''}$, which remains strongly self-transverse.

Fix an integer $m>0$. We can use the same arguments as in the proof of Theorem \ref{T:deform} applied to $N_{\tilde{g}_t}$. Note that here the function
$$\max\{\Area_{\tilde{g}_t}(T); \text{ $T$ is a connected component of $\partial N_{\tilde{g}_t}$}\}$$
may not be differentiable. Remember that this number is the leading coefficient in the cylindrical Weyl law of the min-max widths $\{\tilde{\omega}_p(N_{\tilde{g}_t},{\tilde{g}_t})\}$; see Theorem \ref{T:cylindrical Weyl Law}. Nevertheless it is Lipschitz in $t$ so differentiable almost everywhere and the derivative can easily be estimated.  Apart from that point, the arguments of the proof of Theorem \ref{T:deform} are unchanged. We find for a large $k_0$, a small $t_{k_0}$, a metric $\tilde{g}_{s}$ close to $g_{s}$ for some $s\in [0, t_{k_0}]$, and a large $p_{k_0}$ (we do not need to control how large compared to $k_0$), so that there is a non-degenerate, closed, embedded, minimal surface $\Sigma\subset (N_{\tilde{g}_s},\tilde{g}_{s})$ of Morse index $p_{k_0}$, disjoint from the boundary $\partial N_{\tilde{g}_s}$ such that 
\[ \quad\quad  \big|\frac{1}{||{\Si}||}\int_{{\Si}} \hat{\varphi} d{\Si}-1 \big| \]
is arbitrarily small depending on the choice of $k_0$.
Since $\hat{\varphi}$ is supported in a $1/(2k_0)$-neighborhood of $\mathfrak{X}_{\tilde{g}_s}$, ${\Sigma}$ is mostly supported in a $1/k_0$-neighborhood of $\mathfrak{X}_{\tilde{g}_s}$: in other words
\begin{equation}\label{ff}
\frac{||{\Sigma}\backslash N_{1/k_0}(\mathfrak{X}_{\tilde{g}_s})||}{||{\Sigma}||}
\end{equation}
can be made arbitrarily small if $k_0$ was taken large enough. 

Note that, by the usual Constancy Theorem and by strong-transversality of $\mathfrak{X}_{\tilde{g}_s}$, any stationary varifold $V$ supported on $\mathfrak{X}_{\tilde{g}_s}$ is a union of constant multiples of lifts of  $S_{\tilde{g}_s}$. If moreover $V$ is a stationary varifold in $(N_{\tilde{g}_s}, \tilde{g}_s)$, then by construction of $N_{\tilde{g}_s}$, $V$ has to be supported in the boundary $\partial N_{\tilde{g}_s}$.
So by a compactness argument, the minimal surface $\Sigma$ can be chosen so that $\frac{[{\Sigma}]}{||{\Sigma}||}$ is arbitrarily close to a stationary varifold supported in $\partial N_{\tilde{g}_s}$ if $k_0$ was large enough.

In all these arguments, it can be checked that the perturbation $\{\tilde{g}_t\}$ can be chosen arbitrarily close to $\{g_t\}$ in the $C^\infty$ topology, independently of $k_0$. It means that for some large but uniform $k_0$, $t_{k_0}^{-1}$ and $p_{k_0}$, the above conclusions (i.e. the construction of a $\Sigma$ with some properties) hold for all small enough perturbations $\{\tilde{g}_t\}$. Then by compactness \cite{Sharp17} (letting the perturbations $\{\tilde{g}_t\}$ converging back to $\{g_t\}$), we get an $s\in [0,1]$ and a limit minimal surface $Z$ in $(N_{g_s},g_s)$.  Note that no component of $Z$ is contained in $\partial N_{g_s}$; in fact, if this was not true, a sequence of connected minimal surfaces (in some $(N_{\tilde{g}_s},\tilde{g}_{s})$) would converge to $\partial N_{g_s}$ in the Hausdorff distance (by a standard argument using the monotonicity formula), which contradicts the maximum principle as a neighborhood of $\partial N_{g_s}$ has a mean convex foliation. Moreover,   
$\frac{[Z]}{||Z||}$ is arbitrarily close to a stationary varifold supported in $\partial N_{g_s}$ if $k_0$ was large enough. 
Note that $Z$ could a priori be degenerate, since $g_s$ is in general not bumpy in any sense. However $g_s$ is smooth, equivariant and descends to a metric on $M$.   By a counting argument, we can take a connected component of $Z$, still denoted as $Z$, such that the quantity~\eqref{ff} is arbitrarily small when $k_0$ is large enough. Therefore, we can make sure that the connected immersed minimal surface $\pi(Z)$ satisfies
\[ \mathbf{F}(\frac{[S_{g_s}]}{||S_{g_s}||},\frac{[\pi(Z)]}{||\pi(Z)||}) <\frac{1}{m}. \]

By a conformal perturbation of $g_s$, we can make $Z$ non-degenerate; (see for instance \cite[Proof of Lemma 2.6]{White17}). It only remains to prove that the area and Morse index of $Z$ can be chosen arbitrarily large (if $k_0$ was chosen large enough). {The area of $Z$ has to be large by the maximum principle; indeed if $Z$ had uniformly bounded area as $k_0$ gets larger,  $Z$ would be close to a union of boundary components of $\partial N_{{g}_s}$ in the Hausdorff topology by monotonicity formula, and this contradicts the strict stability of $\partial N_{{g}_s}$.} 

As for the Morse index of $Z$, it follows from \cite[Theorem 1.17]{Chodosh-Ketover-Maximo17}.  The details were written in \cite{Collin-Hauswirth-Mazet-Rosenberg18} for a special case. 
We can sketch the idea as follows. Suppose towards a contradiction that $Z$ cannot be chosen to have arbitrarily large Morse index (when taking $k_0$ large), then we would be able to construct a sequence metrics $g_{s^{(j)}}$ and a sequence of closed embedded minimal hypersurfaces $Z^{(j)}\subset (N_{g_{s^{(j)}}},g_{s^{(j)}})$, such that 
\begin{enumerate}
\item $g_{s^{(j)}}$ converges smoothly to $g$ as $j\to \infty$,
\item $\frac{[Z^{(j)}]}{||Z^{(j)}||}$ converges (with respect to $g$) to a varifold supported in the boundary $\partial N_g$,
\item the area of $Z^{(j)}$ goes to infinity as $j\to \infty$,
\item but the index of $Z^{(j)}$ stays uniformly bounded as $j\to \infty$.
\end{enumerate}
Next by bullet (4), $Z^{(j)}$ converges to a smooth minimal lamination of $(N_g,g)$, and a connected component $\partial N_g$ has to be a leaf of this lamination because of bullet (2) above and the Constancy Theorem.  In \cite{Collin-Hauswirth-Mazet-Rosenberg18} the authors studied bounded index minimal hypersurfaces converging to a minimal lamination staying ``on one side'' of a certain closed minimal leaf. Their arguments imply here that for any $\varepsilon>0$, for $j$ large, there is a closed connected component of $Z^{(j)}$ contained in the $\varepsilon$-tubular neighborhood of $\partial N_{g_{s^{(j)}}}$. This contradicts the maximum principle since $\partial N_{g_{s^{(j)}}}$ is strictly stable with respect to $g_{s^{(j)}}$.

The Fact is proved and this finishes the proof of Theorem \ref{T:3dimension}.
\end{proof}

For closed hyperbolic $3$-manifolds, we get a stronger result:

\begin{corollary} \label{C:hyperbolic}
If $M^3$ admits a hyperbolic metric, then for a generic metric $g$ on $M$, any $\pi_1$-injective closed surface has an area minimizing representative in its homotopy class, which is the scarring limit of a sequence of closed, immersed, minimal surfaces with area and Morse index diverging to infinity.

\end{corollary}
\begin{proof}
If $M$ admits a hyperbolic metric, then by Kahn-Markovi\'{c} \cite{Kahn-Markovic12a,Kahn-Markovic12b}, there are $\pi_1$-injective surfaces in $M$. By Agol \cite{Agol13}, $\pi_1(M)$ is LERF so each of these $\pi_1$-injective surfaces lifts to an embedded $2$-sided surface in a finite cover of $M$; (see for instance \cite{Matsumoto02}). Given a $\pi_1$-injective surface $F$ and any metric $g$ on $M$, we get an immersed area minimizing surface $S\subset (M, g)$ in its homotopy class by Schoen-Yau \cite{Schoen-Yau79} or  Sacks-Uhlenbeck \cite{Sacks-Uhlenbeck82}. Let $S_1$ be an embedded $2$-sided stable lift of $S$ in a finite cover $M_{cover}$. By inspecting the proof of Theorem \ref{T:3dimension}, we see that the arguments apply to $S_1\subset M_{cover}$ so for a generic smooth metric on $M$, there is a sequence of immersed minimal surfaces scarring along an area minimizing surface $S$ in the homotopy class of $F$. By taking a countable intersection of sets of metrics, we obtain a generic set of smooth metrics for which any oriented $\pi_1$-injective surface has an area minimizing representative which is the scarring limit of a sequence of immersed minimal surfaces.
\end{proof}

\appendix

\section{Differentiability of the first eigenvalue}
\label{A:Differentiability of the first eigenvalue}

For the reader's convenience, we give a proof of the known fact that the first eigenvalue of a self-adjoint elliptic operator depends in a differentiable way on the coefficients; (see \cite{Uhlenbeck76} for related results).
Let $\Sigma$ be a smooth closed $n$-manifold.  Fix two integers $r\geq 2$, $k\geq1$. Let $U$ be a smooth Banach manifold and let $\Phi: U \to  \big(C^r(\Sigma)\big)^{n^2+n+1} $ be a $C^k$ map which associates to any $\gamma\in U$ a triple $\Phi(\gamma) = \big((a^{ij})_{1\leq i, j \leq n}, (b^i)_{1\leq i\leq n}, c\big)$, where $a^{ij}, b^i, c\in C^r(\Sigma)$, such that $(a^{ij})>0$ is positive definite and such that the following elliptic operator  
\[ L_{\gamma}u := a^{ij}u_{ij} + b^i u_i + cu\]
is self-adjoint with respect to a volume measure depending on $\gamma\in U$.

A number $\lambda_1(\gamma)$ is the first eigenvalue of $L_{\gamma}$ if and only if 
\[ L_{\gamma} \varphi_1 = -\lambda_1(\gamma) \varphi_1 \]
for some $\varphi_1 \in H^{1}(\Sigma)\cap C^0(\Sigma)$ with $\varphi_1>0$. The eigenfunctions of $L_{\gamma}$ span $L^2(\Sigma)$ and are in $C^r(\Sigma)$.

\begin{lemma}
\label{L:Differentiability of the first eigenvalue}
$\lambda_1(\gamma)$ is a $C^k$ function of $\gamma\in U$.
\end{lemma}
\begin{proof}
Fix a metric $g_0$ on $\Sigma$ and let $S^{r}(\Sigma)$ be the space of functions $u\in H^{r}(\Sigma)$ with $\|u\|_{L^2(\Sigma,g_0)}=1$, where the $L^2$-norm is computed with $g_0$. Note that $H^{r}(\Sigma)$ does not depend on the metric. Consider the operator $T: U\times \R \times S^{r}(\Sigma) \to H^{r-2}(\Sigma)$ defined by
\[ T\big(\gamma, \mu, u\big) = L_{\gamma}u + \mu u. \]
Since $\Phi:U \to \big(C^r(\Sigma)\big)^{n^2+n+1}$ is a $C^k$ map by assumption, $T$ is also $C^k$-differentiable in its variables.
We have $T\big(\gamma, \lambda, \varphi\big)=0$ if and only if $\lambda$ is an eigenvalue of $L_{\gamma}$ and $\varphi$ is an associated normalized eigenfunction. An eigenvalue is simple if the associated eigenspace is one dimensional. It is known that the first eigenvalue of $L_{\gamma}$ is always simple. Given $\gamma_1 \in U$, consider the first eigenvalue $\lambda_1$ of $L_{\gamma_1}$ and a normalized positive first eigenfunction $\varphi_1\in S^{r}(\Sigma)$, and then consider the following differential
\[ D_{(\mu, u)} T|_{(\gamma_1, \lambda_1, \varphi_1)}(s, v) = L_{\gamma_1}v + \lambda_1 v + s \varphi_1, \]
where $s\in \mathbb R$, $v\in \textrm{span}\{\varphi_1\}^{\perp_{g_0}} = \{ u\in H^{r}(\Sigma), \int_{\Sigma} u \varphi_1d\text{vol}_{g_0}=0\}$. We know that the null space of $L_{\gamma_1}+\lambda_1$ is $\textrm{span}\{\varphi_1\}$. Moreover since  $L_{\gamma_1} +\lambda_1$ is self adjoint for a volume measure $\nu$, and since $\lambda_1$ is simple, the image of $\textrm{span}\{\varphi_1\}^{\perp_{g_0}}$ by this operator is  $\{ u\in H^{r-2}(\Sigma), \int_{\Sigma} u \varphi_1 d\nu=0\}$.
Hence
\[ D_{(\mu, u)} T|_{(\gamma_1, \lambda_1, \varphi_1)}: \mathbb R\times \textrm{span}\{\varphi_1\}^{\perp_{g_0}} \to H^{r-2}(\Sigma) \]
is an isomorphism. By the Implicit Function Theorem, near $\gamma_1$, there exist $C^k$ maps:
\[ \gamma \to \lambda(\gamma)\in \R \quad \text{and} \quad \gamma\to \varphi(\gamma)  \in H^{r}(\Sigma)\]
with 
\[ \lambda(\gamma_1)=\lambda_1 \quad \text{and} \quad \varphi(\gamma_1)= \varphi_1, \]
such that $L_{\gamma}\varphi(\gamma) = -\lambda(\gamma)\varphi(\gamma)$. Note that $\varphi(\gamma)$ is also positive when $\gamma$ is close enough to $\gamma_1$, so $\lambda(\gamma)$ is indeed the first eigenvalue of $L_\gamma$. The conclusion then follows.
\end{proof}

\bibliography{refs}
\bibliographystyle{plain}

\end{document}